\documentclass[12pt]{article}

\usepackage{amssymb}
\usepackage{amsmath}
\usepackage{amsthm}
\usepackage{hyperref}
\newtheorem{theorem}{Theorem}[section]
\newtheorem{proposition}[theorem]{Proposition}
\newtheorem{lemma}[theorem]{Lemma}
\newtheorem{remark}[theorem]{Remark}
\usepackage{authblk}
\usepackage[square, numbers]{natbib}
\usepackage{graphicx}

\hypersetup{
    colorlinks=true,
    linkcolor=blue,
    filecolor=blue,      
    urlcolor=cyan,
    citecolor = blue
}

\title{Edge detection with polynomial frames on the sphere}
\date{}

\author{Frederic Schoppert \\ \href{mailto:f.schoppert@uni-luebeck.de}{f.schoppert@uni-luebeck.de} }

\affil{Institute of Mathematics, University of L{\"u}beck, Ratzeburger Allee 160, 23562 L{\"u}beck, Germany}

\begin{document}

\maketitle

\begin{abstract}
In a recent article, we have shown that a variety of localized polynomial frames, including isotropic as well as directional systems, are suitable for detecting jump discontinuities along circles on the sphere. More precisely, such edges can be identified in terms of their position and orientation by the asymptotic decay of the frame coefficients in an arbitrary small neighborhood. In this paper, we will extend these results to discontinuities which lie along general smooth curves. In particular, we prove upper and lower estimates for the frame coefficients when the analysis function is concentrated in the vicinity of such a singularity. The estimates are given in an asymptotic sense, with respect to some dilation parameter, and they hold uniformly in a neighborhood of the smooth curve segment under consideration.
\end{abstract}



\section{Introduction}\label{sec1}
The classical orthonormal basis of spherical harmonics $Y_n^k$ plays a fundamental role in signal analysis on the sphere $\mathbb{S}^2$. It allows for a unique representation
\begin{equation*}
f = \sum_{n=0}^\infty\sum_{k=-n}^n \langle f, Y_n^k \rangle Y_n^k
\end{equation*}
of any signal $f \in L^2(\mathbb{S}^2)$ with respect to this basis and $f$ is completely determined by its Fourier coefficients $\langle f, Y_n^k \rangle$. However, in general it is not obvious how one can extract some desired information about $f$ from these inner products. Here, the study of localized features proves to be particularly difficult, since the spherical harmonics are not concentrated in space and, thus, the coefficients $\langle f, Y_n^k \rangle $ are global quantities. One way to tackle this problem is to consider frames consisting of localized functions. Indeed, if $(\Psi_j)_{j\in J}$ is such a system, the frame coefficients $\langle f, \Psi_j\rangle$, which, again, contain all the information about the signal $f$, give access to a position-based analysis. Furthermore, if the $\Psi_j$ are polynomials, then the inner products $\langle f, \Psi_j\rangle$ can be computed as linear combinations of the Fourier coefficients $\langle f, Y_n^k \rangle$. In this paper, we prove that, for a wide class of localized polynomial frames on the sphere, the frame coefficients display accurately the positions and orientations of jump discontinuities which lie along sufficiently smooth curves.

In the one-dimensional setting, the problem of detecting singularities with localized polynomial frames has been well investigated (see e.g.\ \citep{ bib7, bib12, bib13, bib14}). As a more recent development, the bivariate case has also received a lot of attention (see e.g.\ \citep{bib5, bib6, bib8,bib17, bib18}). Here, the problem is more complex since singularities can lie along curves and, consequently, directional frames are needed to identify such features both in terms of their position and orientation. On the sphere $\mathbb{S}^2$, this matter was first investigated in the previous article \citep{bib25}, where the results are limited to singularities on circles. In this present paper, we will investigate the detection of jump discontinuities that lie along more general curves. Our approach is based on the study of frame coefficients with respect to indicator functions $\mathbf{1}_A$, where $A\subset \mathbb{S}^2$ is a simply connected region with a sufficiently smooth boundary $\partial A$. Here, we take advantage of the fact that, locally, smooth curves are approximated sufficiently well by suitable circles. This allows us to transfer some of our previous results to this more general setting. Specifically, for the class of directional wavelets $\Psi_{\scriptscriptstyle K}^{\scriptscriptstyle N}$, introduced and studied in \citep{bib9, bib10, bib21}, we show that the frame coefficients $\langle \mathbf{1}_A, \mathcal{D}(\mathbf{x}, \mathbf{r} ) \Psi_{\scriptscriptstyle K}^{\scriptscriptstyle N} \rangle$ peak  when the distance $d(\mathbf{x}, \partial A)$ between the center $\mathbf{x}$ of the rotated wavelet $\mathcal{D}(\mathbf{x}, \mathbf{r} ) \Psi_{\scriptscriptstyle K}^{\scriptscriptstyle N}$ and the boundary is small enough and $\mathcal{D}(\mathbf{x}, \mathbf{r} ) \Psi_{\scriptscriptstyle K}^{\scriptscriptstyle N}$ exhibits roughly the same orientation $\mathbf{r}$ as the closest edge segment. Indeed, in \autoref{thm1} we will show that there exists some $N_0 \in \mathbb{N}$ and a nonempty open interval $I\subset (0, \infty)$ such that
\begin{equation*}
\lvert \langle \mathbf{1}_A, \mathcal{D}(\mathbf{x}, \mathbf{r} ) \Psi_{\scriptscriptstyle K}^{\scriptscriptstyle N} \rangle \rvert \geq c_1>0 \, \quad \text{if } N d(\mathbf{x}, \partial A) \in I,\; N\geq N_0,
\end{equation*}
provided that the orientation $\mathbf{r}$ does not deviate too much from the orientation of the nearest edge. Here, larger values of the dilation parameter $N$ result in a better localization of the corresponding wavelet. Additionally, we derive an upper bound of the form
\begin{equation*}
\lvert \langle \mathbf{1}_A, \mathcal{D}(\mathbf{x}, \mathbf{r}) \Psi_{ \scriptscriptstyle K}^{\scriptscriptstyle N} \rangle \rvert \leq c_2(1+N d( \mathbf{x}, \partial A))^{-\ell},
\end{equation*}
where $\ell$ is some positive integer. The constants $c_1, c_2$ can be chosen independent of $\mathbf{x}, \mathbf{r}$ and $N$, as long as the wavelet is positioned in some neighborhood of the curve segment  under consideration. Another insightful result about the frame coefficients will be presented in \autoref{thm2}, consisting of an asymptotic formula in which the leading term is given explicitly as function depending essentially on the difference in position and orientation between the wavelet and the closest boundary segment. Even though all our findings are formulated in terms of the continuous wavelet transform, the implications for corresponding discretized systems, as discussed for example in \citep{bib10, bib21}, are obvious.

The remainder of this paper is organized as follows. In \autoref{sec2}, we go over some basic notations and give an explicit basis of spherical harmonics. \autoref{sec3} serves as a short introduction to the system of directional wavelets \citep{bib9, bib10, bib21}. Furthermore, we formulate a new localization bound which has a simple proof and slightly improves the bound given in \citep{bib9}. In \autoref{sec4}, we present our main results, which are derived in two primary steps. First, we prove an auxiliary lemma which states that, locally, smooth curves are approximated sufficiently well by suitable circles. Subsequently, we combine said lemma with our previous results in \citep{bib25} to obtain the desired asymptotic estimates. Finally, in \autoref{sec5}, we present some numerical experiments which illustrate our theoretical findings.

\section{Preliminaries}\label{sec2}
We consider the space $\mathbb{R}^3$ equipped with the euclidean norm $\left\| \cdot \right\|_2$, which is induced by the inner product  $\langle \mathbf{x}, \mathbf{y} \rangle_2 = \mathbf{x}^\top \mathbf{y}$ for $\mathbf{x}, \mathbf{y} \in \mathbb{R}^3$. The vectors
\begin{equation*}
\mathbf{e}_1 = \begin{pmatrix}
1 \\
0 \\
0
\end{pmatrix}, \quad \mathbf{e}_2 = \begin{pmatrix}
0 \\
1 \\
0
\end{pmatrix}, \quad \mathbf{e}_3 = \begin{pmatrix}
0 \\
0 \\
1
\end{pmatrix}
\end{equation*}
form the canonical basis of $\mathbb{R}^3$. As usual, we will denote the two-dimensional unit sphere by $\mathbb{S}^2 = \left\{ \mathbf{x}\in \mathbb{R}^3 : \left\| \mathbf{x} \right\|_2 = 1 \right \}$ and refer to $\mathbf{e}_3$ as the north pole of $\mathbb{S}^2$. Each point $\mathbf{x} \in \mathbb{S}^2$ can be identified in terms of its latitude $\theta \in [0, \pi]$ and, not necessarily unique, longitude $\varphi \in [0, 2\pi)$ through
\begin{equation}\label{polarkoord.}
\mathbf{x} (\theta, \varphi) = \begin{pmatrix}
\sin \theta \, \cos \varphi \\
\sin \theta \, \sin \varphi \\
\cos \theta
\end{pmatrix}.
\end{equation}
The geodesic distance between two points $\mathbf{x}, \mathbf{y} \in \mathbb{S}^2$ is given by
\begin{equation*}
d(\mathbf{x}, \mathbf{y})= \arccos\left( \langle \mathbf{x}, \mathbf{y}\rangle_2 \right).
\end{equation*}
It is well known that $d$ is a metric on $\mathbb{S}^2$ and that
\begin{equation}\label{normeq}
\frac{2}{\pi} d(\mathbf{x}, \mathbf{y}) \leq \|\mathbf{x} - \mathbf{y} \|_2 \leq d(\mathbf{x}, \mathbf{y}) \quad \text{for all } \mathbf{x}, \mathbf{y}\in \mathbb{S}^2.
\end{equation}
The set
\begin{equation*}
C(\mathbf{z}, \phi) = \left\{ \mathbf{x}\in \mathbb{S}^2 : d(\mathbf{x}, \mathbf{z}) < \phi  \right\}   
\end{equation*}
is called a spherical cap with center $\mathbf{z} \in \mathbb{S}^2$ and opening angle \(\phi \in \left(0, \pi \right)\). Its boundary
\begin{equation*}
\partial C(\mathbf{z}, \phi) = \left\{  \mathbf{x}\in \mathbb{S}^2 :d( \mathbf{x}, \mathbf{z} ) =  \phi  \right\}   
\end{equation*}
constitutes a circle on the sphere. Furthermore, in the case that $\phi = \pi/2$, we call $\partial C(\mathbf{z}, \pi/2)$ a great circle. Finally, if $\mathbf{x}\in \mathbb{S}^2$ and $A\subset \mathbb{S}^2$ we define
\begin{equation*}
d(\mathbf{x}, A) = \inf_{\mathbf{a}\in A} d(\mathbf{x}, \mathbf{a}).
\end{equation*}

In the following, we consider the Hilbert space $L^2(\mathbb{S}^2)$ of square integrable functions on $\mathbb{S}^2$ with the corresponding inner product
\begin{align*}
\langle f, g \rangle & = \int_{\mathbb{S}^2} f(\mathbf{x})\, \overline{g(\mathbf{x})} \, \mathrm{d}\omega(\mathbf{x}) \\
& =  \int_0^{2\pi} \int_{0}^\pi f(\theta, \varphi) \, \overline{g(\theta, \varphi)} \, \sin \theta \, \mathrm{d}\theta \, \mathrm{d}\varphi \quad \text{for} \; f, g \in L^2(\mathbb{S}^2),
\end{align*}
where $\mathrm{d}\omega(\mathbf{x}) = \sin \theta \, \mathrm{d}\theta \, \mathrm{d}\varphi$ is the usual surface element. A common choice for a polynomial orthonormal basis consists of the functions $Y_n^k \colon \mathbb{S}^2 \rightarrow \mathbb{C}$, where $Y_n^k$ denotes the spherical harmonic of degree $n\in \mathbb{N}_0$ and order $k\in \left\{ -n,-n+1, ..., n \right \}$ defined as
\begin{equation}\label{spherical harmonic}
Y_{n}^k(\theta, \varphi) = (-1)^k \, \sqrt{\frac{2 n + 1}{4 \pi}\frac{(n-k)!}{(n +k)!}}\, P_n^k(\cos\theta) \exp(\mathrm{i} k \varphi).
\end{equation}
The associated Legendre polynomials $P_n^k \colon [-1, 1]\rightarrow \mathbb{R}$ used in \eqref{spherical harmonic} are given by
\begin{equation}\label{assoc. Legendre def}
P_n^k (x) = (1-x^2)^{k/2}\, \frac{\mathrm{d}^{ k}}{\mathrm{d}x^{k}} P_n(x)  \quad \text{for } k \geq 0
\end{equation}
and
\begin{equation*}\label{assozLegendreNegativ}
P_n^{k}(x) = (-1)^k \, \frac{(n+k)!}{(n-k)!} \, P_n^{-k}(x)   \quad \text{for } k < 0,
\end{equation*}
where the Legendre polynomial $P_n \colon [-1,1]  \rightarrow \mathbb{R}$ can be defined via the Rodrigues formula
\begin{equation*}
P_n (x) = \frac{1}{2^n n !} \frac{\mathrm{d}^{n}}{\mathrm{d}x^{n}} \Big( (x^2-1)^n \Big) .
\end{equation*}
If $f$ is a continuous function defined on the sphere, we use the standard notation
\begin{equation*}
\| f \|_\infty = \sup_{\mathbf{x}\in \mathbb{S}^2} \lvert f(\mathbf{x}) \rvert.
\end{equation*}

Functions in $L^2(\mathbb{S}^2)$ can be rotated using elements in the rotation group $SO(3)=\{ \mathbf{R}\in \mathbb{R}^{3\times 3} : \mathbf{R}^{-1}=\mathbf{R}^{\top} \text{ and } \det (\mathbf{R})=1  \}$. In this paper, we will often refer to a parameterization of rotation matrices $\mathbf{R}(\mathbf{x}, \mathbf{r}) \in SO(3)$ in terms of elements in the unit tangent bundle  $UT\mathbb{S}^2 = \{ (\mathbf{x}, \mathbf{r}) : \mathbf{x}\in \mathbb{S}^2, \; \mathbf{r}\in T_\mathbf{x}\mathbb{S}^2 \text{ and } \| \mathbf{r} \|_2 = 1 \}$, where $T_\mathbf{x}\mathbb{S}^2$ denotes the tangent space of $\mathbb{S}^2$ at $\mathbf{x}$. Here, the correspondence is given by the bijection $SO(3)\rightarrow UT\mathbb{S}^2$,
\begin{equation*}
\mathbf{R}\mapsto (\mathbf{R}\mathbf{e}_3, \mathbf{R}\mathbf{e}_1).
\end{equation*}
Hence, $\mathbf{R}(\mathbf{x}, \mathbf{r})$ is the unique rotation matrix mapping $\mathbf{e}_3$ to $\mathbf{x}$ and $\mathbf{e}_1$ to $\mathbf{r}$. In this context, we define the rotation operator $\mathcal{D}(\mathbf{x}, \mathbf{r}) \colon L^2(\mathbb{S}^2)\rightarrow L^2(\mathbb{S}^2)$ by
\begin{equation}\label{rotation operator}
\mathcal{D}(\mathbf{x}, \mathbf{r})f(\mathbf{y}) = f([\mathbf{R}(\mathbf{x}, \mathbf{r})]^{-1}\mathbf{y}) \quad \text{for all } \mathbf{y}\in \mathbb{S}^2.
\end{equation}
In the literature, elements of the rotation group are also commonly represented as matrices $\mathbf{R}(\alpha, \beta, \gamma)$ in terms of the Euler angles $\alpha, \gamma \in [0, 2\pi)$ and $\beta \in [0, \pi]$, where $\mathbf{R}(\alpha, \beta, \gamma) = \mathbf{R}_{\mathbf{e}_3}(\alpha) \mathbf{R}_{\mathbf{e}_2}(\beta) \mathbf{R}_{\mathbf{e}_3}(\gamma)$ and $\mathbf{R}_{\mathbf{e}_i}(\xi)$ is the $3\times 3$ rotation matrix which rotates around the $\mathbf{e}_i$-axis by an angle $\xi$. The rotation operator can then also be given in terms of $\alpha, \beta$ and $\gamma$ through
\begin{equation*}
\mathcal{D}(\alpha, \beta,\gamma)f(\mathbf{y}) = f([\mathbf{R}(\alpha, \beta, \gamma)]^{-1}\mathbf{y}) \quad \text{for all } \mathbf{y}\in \mathbb{S}^2.
\end{equation*}
While for the most part we will use the unit tangent bundle parameterization of $SO(3)$, occasionally we will also refer to the Euler angle representation. A connection between these two approaches is given by
\begin{equation*}
\mathbf{R}(\alpha, \beta, \gamma) = \mathbf{R}(\mathbf{x}(\beta, \alpha), \mathbf{R}(\alpha, \beta, \gamma)\mathbf{e}_1),
\end{equation*}
with $\mathbf{x}(\beta, \alpha)\in \mathbb{S}^2$ and $\mathbf{R}(\alpha, \beta, \gamma)\mathbf{e}_1 \in T_{\mathbf{x}(\beta, \alpha)}\mathbb{S}^2$.

Finally, for $\mathbf{x}\in \mathbb{S}^2$ we define $d_\mathbf{x}(\mathbf{r}, \mathbf{s}) \in [0, 2\pi)$ to be the order dependent angle between two vectors $\mathbf{r}, \mathbf{s}\in T_\mathbf{x}\mathbb{S}^2$, i.e., $\mathbf{r} = \mathbf{R}_\mathbf{x}(d_\mathbf{x}(\mathbf{r}, \mathbf{s}))\mathbf{s}$, where $\mathbf{R}_\mathbf{x}(d_\mathbf{x}(\mathbf{r}, \mathbf{s}))$ is the $3\times 3$ matrix performing a counterclockwise rotation around the $\mathbf{x}$ axis by the angle $d_\mathbf{x}(\mathbf{r}, \mathbf{s})$.

\section{Localized polynomial frames on the sphere}\label{sec3}
In the series of papers \citep{ bib9, bib10, bib21} the authors introduced and investigated systems of so called directional wavelets on the sphere, which constitute a wide class of isotropic and non-isotropic polynomial frames. These frame functions are highly localized in space and therefore provide a powerful tool for position-based analyses. Following closely the construction and notation in \citep{bib9}, we define the directional wavelets in terms of their Fourier coefficients by
\begin{equation}\label{directional wavelets}
\Psi_{\scriptscriptstyle K}^{\scriptscriptstyle N} = \sum_{n=0}^\infty \sum_{k=-n}^n \sqrt{\frac{2n+1}{8\pi^2}} \, \kappa\!\left( \frac{n}{N} \right) \zeta_{n, k}^{\scriptscriptstyle K} \, Y_n^k,
\end{equation}
where $\kappa$ is a window function satisfying $\kappa \in C^{q}([0, \infty))$ for some $q \in \mathbb{N}$ and $\emptyset \neq \text{supp}(\kappa)\subset [1/2, 2]$. The spatial localization of $\Psi_{\scriptscriptstyle K}^{\scriptscriptstyle N}$ is controlled by the dilation parameter $N\in \mathbb{N}$, where larger values correspond to higher concentrated wavelet functions. As suggested in \citep{bib9}, for $K \in \mathbb{N}$ we define the directionality component by
\begin{equation*}
\zeta_{n, k}^{\scriptscriptstyle K} = \begin{cases}  \displaystyle \eta\, \nu \, \sqrt{\frac{1}{2^p} \binom{p}{\frac{p-k}{2}}}, \quad  & \text{if}\; \lvert k  \rvert \leq K-1, \\
0,  & \text{if}\;\lvert k  \rvert \geq K,
\end{cases}
\end{equation*}
where
\begin{equation*}
\eta = \begin{cases}
1, \quad & \text{if}\; K -1 \; \text{even},\\
\mathrm{i},   & \text{if}\; K-1 \; \text{odd},    
\end{cases} \quad  \quad \quad
\nu = \begin{cases}
1, \quad & \text{if}\; K+k \; \text{odd},\\
0, & \text{if}\; K+k \;  \text{even}  
\end{cases} \quad 
\end{equation*}
and
\begin{equation*}
p =\begin{cases}
\min (K-1, n-1), \quad & \text{if}\; K+n \; \text{even},\\
\min (K-1, n), & \text{if}\; K+n \; \text{odd}.
\end{cases}
\end{equation*}
This results in $\langle \Psi_{ \scriptscriptstyle K}^{\scriptscriptstyle N}, Y_n^k\rangle = 0$ for $\lvert k \rvert \geq K$, i.e., the directional wavelets are band-limited with respect to the azimuthal frequency. Furthermore, $\Psi_{\scriptscriptstyle K}^{\scriptscriptstyle N}$ becomes increasingly more directional as $K$ gets large. Since $\zeta_{n, k}^{\scriptscriptstyle K}$ does not depend on $n$ for $n \geq K$, we will also write $ \zeta_{k}^{\scriptscriptstyle K}$ instead of $\zeta_{n, k}^{\scriptscriptstyle K}$ in this case. We will see that the function
\begin{equation*}
\chi_{\scriptscriptstyle K}(\gamma)= \overline{\eta} \sum_{k = 1-K}^{K-1}\zeta_{k}^{\scriptscriptstyle K} \exp(\mathrm{i} k \gamma ),
\end{equation*}
which is essentially the inverse Fourier transform of the directionality component, characterizes the directional sensitivity of $\Psi_{\scriptscriptstyle K}^{\scriptscriptstyle N}$ with respect to detecting discontinuities. Indeed, in the special case where the singularities lie along circles this was previously observed in \citep[Theorem~3.1]{bib25}. The functions $\chi_{\scriptscriptstyle K}$ are visualized in \autoref{Fig1} for different values of $K$.

\begin{figure}[t]
\center
\includegraphics[width=1\textwidth]{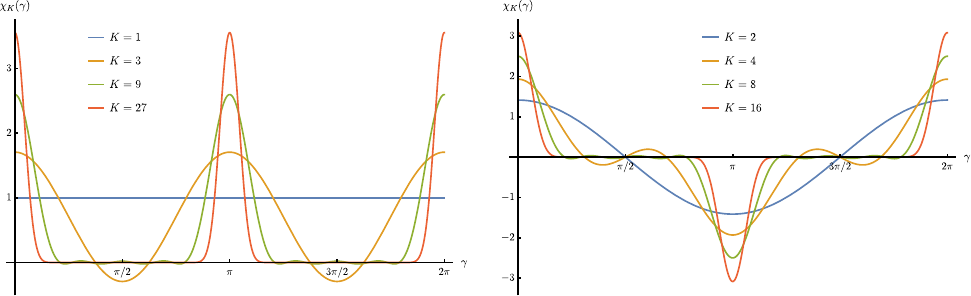}
\caption{Directionality functions $\chi_{\scriptscriptstyle K}$ for odd (left) and even (right) values of $K$}\label{Fig1}
\end{figure}

The localization bound for directional wavelets derived in \citep{bib9} plays a central role in the proof of our main results. Therefore, we would like to reproduce the corresponding statement here once again, using our slightly different notation. Moreover, we give an alternative simple proof, which is based on \cite[Theorem~2.6.7]{bib3}, and our result improves the existing bound stated in \citep{bib9}. Here, and in the remainder of the article, all occurring constants $c, c_0, c_1, ...$ only depend on the indicated parameters and their values might change with each appearance.

\begin{proposition}
Let $\Psi_{\scriptscriptstyle K}^{\scriptscriptstyle N}$ be the directional wavelet defined in \eqref{directional wavelets} with $\kappa \in C^{3q-1}([0, \infty))$ for some $q \in \mathbb{N}$. 
Then there exists a constant $c = c(\kappa, K, q)>0$ such that
\begin{equation}\label{wavelets bound 1}
\lvert \Psi_{\scriptscriptstyle K}^{\scriptscriptstyle N}(\theta, \varphi) \rvert \leq \frac{c N^{2}}{(1+N \theta)^{q}} \sum_{k=0}^{K-1}(N \sin \theta)^k \quad \quad  \text{for all } \theta \in [0, \pi], \; \varphi \in [0, 2\pi).
\end{equation}
In particular,
\begin{equation}\label{wavelets bound 2}
\lvert \Psi_{\scriptscriptstyle K}^{\scriptscriptstyle N}(\theta, \varphi) \rvert \leq \frac{c N^{2}}{(1+N \theta)^{q+1-K}} \quad \quad \text{for all } \theta \in [0, \pi], \; \varphi \in [0, 2\pi).
\end{equation}
\end{proposition}
\begin{proof}
For $N\geq 2K$, the occurring directionality components $\zeta_{n, k}^{\scriptscriptstyle K}$ in \eqref{directional wavelets} are independent of $n$, i.e., $\zeta_{n, k}^{\scriptscriptstyle K} = \zeta_k^{\scriptscriptstyle K}$. By using the symmetry $Y_n^{-k} = (-1)^k\overline{Y_n^k}$, as well as \eqref{spherical harmonic} and \eqref{assoc. Legendre def}, we obtain
\begin{align*}
\Psi_{\scriptscriptstyle K}^{\scriptscriptstyle N}(\theta, \varphi) = \frac{1}{\sqrt{2\pi}} & \sum_{k=0}^{K-1} \left( \zeta_{ k}^{\scriptscriptstyle K} \exp(\mathrm{i}k(\varphi + \pi)) + (1-\delta_{k, 0})\zeta_{-k}^{\scriptscriptstyle K} \exp(-\mathrm{i}k\varphi) \right)(\sin \theta)^k \\
& \quad \quad \times \sum_{n=0}^\infty\frac{2n+1}{4\pi}\kappa\!\left(\frac{n}{N}\right)\sqrt{\frac{(n-k)!}{(n+k)!}} \,   P_n^{(k)}(\cos \theta),
\end{align*}
where
\begin{equation*}
\delta_{k, j} = \begin{cases}
1, \quad &\text{if } k=j,\\
0, &\text{if } k \neq j.
\end{cases}
\end{equation*}
In the following, we will fix $k$ and derive an upper bound for the above inner sum. For any $p \in \mathbb{N}$ there are constants $c_0, c_1, ..., c_{p-1}$ such that
\begin{equation*}
\sqrt{\frac{(n-k)!}{(n+k)!}} = n^{-k} \left( \sum_{j=0}^{p-1}c_j n^{-j} + \mathcal{O}(n^{-p}) \right).
\end{equation*}
Thus, we have
\begin{align*}
\sum_{n=0}^\infty & \frac{2n+1}{4\pi}\kappa\!\left(\frac{n}{N}\right)\sqrt{\frac{(n-k)!}{(n+k)!}} \, P_n^{(k)}(\cos \theta)\\
& \quad\quad = \sum_{j=0}^{p-1}c_j \sum_{n=0}^\infty\frac{2n+1}{4\pi}\kappa\!\left(\frac{n}{N}\right)n^{-k-j} \,  P_n^{(k)}(\cos \theta) + R_N(\cos \theta),
\end{align*}
where
\begin{equation*}
\lvert R_N(\cos \theta) \rvert \leq c_p \sum_{n=0}^\infty \frac{2n+1}{4\pi} \kappa\!\left(\frac{n}{N} \right) n^{-k-p} \,  \bigg\lvert  P_n^{(k)}(\cos \theta)\bigg\rvert.
\end{equation*}
The classical Markov-inequality for algebraic polynomials yields
\begin{equation*}
 \bigg\lvert   P_n^{(k)}(\cos \theta)\bigg\rvert \leq c_k \,n^{2k} \sup_{x \in [-1,1]} \lvert P_n(x) \rvert = c_k\, n^{2k}
\end{equation*}
and consequently $\lvert R_N(\cos \theta)\rvert \leq c_p N^{k+2-p}$. It will be sufficient to choose $p=k+2+q$. Additionally,
\begin{align*}
&\sum_{n=0}^\infty\frac{2n+1}{4\pi}\kappa\!\left(\frac{n}{N}\right)n^{-k-j} \, P_n^{(k)}(\cos \theta) \\
& \quad \quad \quad = N^{-k-j}  \sum_{n=0}^\infty\frac{2n+1}{4\pi}\, \kappa\!\left(\frac{n}{N}\right) \!\left(\frac{n}{N}\right)^{\!-k-j} \, P_n^{(k)}(\cos \theta) \\
& \quad \quad \quad = N^{-k-j} \sum_{n=0}^\infty\frac{2n+1}{4\pi}\, h\!\left(\frac{n}{N}\right)  P_n^{(k)}(\cos \theta)
\end{align*}
with $h(t) = \kappa(t)t^{-k-j}$ and $h\in C^{3q-1}([0, \infty))$. Now, we can apply \cite[Theorem~2.6.7]{bib3} and obtain
\begin{equation*}
\left\vert \sum_{n=0}^\infty\frac{2n+1}{4\pi}\, h\!\left(\frac{n}{N}\right)  P_n^{(k)}(\cos \theta)\right\vert\leq \frac{c_q N^{2+2k}}{(1+N \theta)^q}.
\end{equation*}
This completes the proof of the estimate \eqref{wavelets bound 1}. The inequality in \eqref{wavelets bound 2} is now a simple consequence. Indeed, by \eqref{wavelets bound 1} it suffices to show that
\begin{equation*}
\frac{ (\sin \theta)^k N^{k}}{(1+N \theta)^{K-1}} \leq c, \quad k=0, 1, ..., K-1.
\end{equation*}
This, however, is obvious in both cases $\theta\leq \pi/2$, where we have $\sin \theta \leq \theta$, and $\theta> \pi/2$, where $(1+N\theta) \geq N$.
\end{proof}

\noindent If $K=1$, the corresponding wavelets
\begin{equation}\label{wavelelets K=1}
\Psi_{\scriptscriptstyle 1}^{\scriptscriptstyle N}(\theta, \varphi) = \frac{1}{\sqrt{2\pi}}\sum_{n=0}^\infty \frac{2n+1}{4 \pi} \, \kappa\!\left( \frac{n}{N} \right) P_n(\cos \theta)
\end{equation}
are isotropic and they can be interpreted as univariate polynomials on the interval $[-1,1]$. Such systems are characterized by their simple structure and have been well investigated in the literature (see e.g. \citep{bib3, bib11,bib14, bib27, bib15, bib16}). In this case, the upper bound \eqref{wavelets bound 2} can be derived under less strict smoothness conditions on the window function $\kappa$. A proof of this fact can be found in \cite[Theorem~1]{bib16}. In our notation, this statement reads as follows.
\begin{proposition}
Let $\Psi_{\scriptscriptstyle 1}^{\scriptscriptstyle N}$ be the directional wavelet defined in \eqref{wavelelets K=1} with $\kappa \in C^{q}([0, \infty))$ for some $q \in \mathbb{N}$ and  $\text{supp}(\kappa)\subset [1/2, 2]$. Then there exists some constant $c = c(\kappa, q)>0$ such that
\begin{equation*}\label{boundK=1}
\lvert \Psi_{\scriptscriptstyle 1}^{\scriptscriptstyle N}(\theta, \varphi) \rvert \leq \frac{c N^{2}}{(1+N \theta)^q} \quad \text{for all } \theta \in [0, \pi], \; \varphi \in [0, 2\pi).
\end{equation*}
\end{proposition}
In summary, the directional wavelets $\Psi_{\scriptscriptstyle K}^{\scriptscriptstyle N}$ defined in \eqref{directional wavelets} are localized at the north pole $\mathbf{e}_3$ and get more and more concentrated as the dilation parameter $N$ increases. The directional sensitivity of  $\Psi_{\scriptscriptstyle K}^{\scriptscriptstyle N}$ is determined by $K$, where larger values correspond to wavelets with a higher directionality. Furthermore, motivated by \autoref{Fig2}, we will say that $\Psi_{\scriptscriptstyle K}^{\scriptscriptstyle N}$ points in the direction $\mathbf{e}_1$.

In the series of papers \citep{bib9, bib10, bib21} the authors discussed how one can construct frames by considering rotated versions of the $\Psi_{\scriptscriptstyle K}^{\scriptscriptstyle N}$ and adding a suitable scaling function. In this article, the concept of rotated wavelets also plays a central role and the operator $\mathcal{D}(\mathbf{x}, \mathbf{r})$ defined in \eqref{rotation operator} has an easy interpretation. Indeed, as visualized in \autoref{Fig2}, the function $\mathcal{D}(\mathbf{x}, \mathbf{r})\Psi_{\scriptscriptstyle K}^{\scriptscriptstyle N}$ is localized at $\mathbf{x}$ and points in the direction $\mathbf{r}$.
\begin{figure}
\centering
\includegraphics[width = 0.4\textwidth]{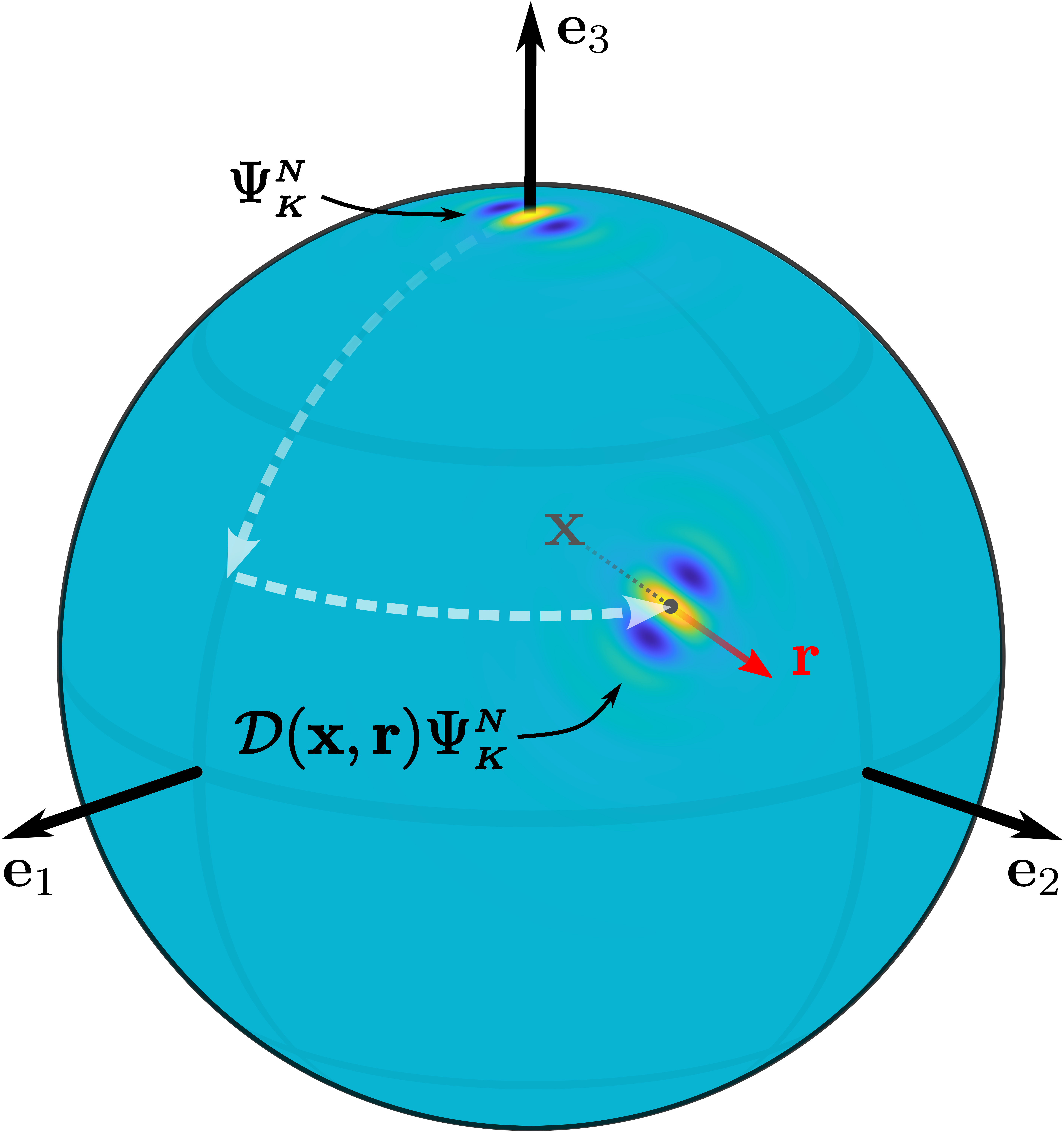}
\caption{Rotation of a localized wavelet function using the operator $\mathcal{D}(\mathbf{x}, \mathbf{r}) $}\label{Fig2}
\end{figure}

\section{Edge detection}\label{sec4}
We begin by describing the situation to which we will refer throughout this section. Let $A \subset \mathbb{S}^2$ be a set whose boundary $\partial A$ is the image of a closed unit-speed curve
\begin{equation*}
\mathbf{v}(t), \quad t \in [0, L],
\end{equation*}
of finite length $L>0$, where $\mathbf{v}$ is injective except for $\mathbf{v}(0) = \mathbf{v}(L)$. We assume that there exists an interval $[a, b]\subset [0, L]$, $a<b$, such that $\mathbf{v}\in C^3([a, b])$. Also, we will assume that $\mathbf{v}$ has a positive orientation, such that for each $p \in [a, b]$ we have
\begin{equation*}
\frac{\mathbf{v}(p)+\varepsilon \mathbf{R}_{\mathbf{v}(p)}(\pi/2)\mathbf{v}'(p) }{\|\mathbf{v}(p)+\varepsilon \mathbf{R}_{\mathbf{v}(p)}(\pi/2)\mathbf{v}'(p) \|_2} \in A,
\end{equation*}
provided that $\varepsilon>0$ is small enough. For $t \in [a, b]$, the curvature $\|\mathbf{v}''(t)\|_2$ of $\mathbf{v}$ at $t$ will play an important part in our calculations. It holds that $\|\mathbf{v}''(t)\|_2\geq 1$, which can be verified as follows. Since $\mathbf{v}$ is a unit speed curve on the sphere, we have $\langle \mathbf{v}(t), \mathbf{v}(t)\rangle_2 = \langle \mathbf{v'}(t), \mathbf{v'}(t)\rangle_2 =1$ for all $t \in [a, b]$. Thus, differentiating $\langle \mathbf{v}(t), \mathbf{v}(t)\rangle_2 \equiv 1$ twice on both sides, we obtain $\langle \mathbf{v}''(t), \mathbf{v}(t) \rangle_2 +1 = 0$. Consequently, the Cauchy-Schwarz inequality yields $1 = \lvert \langle \mathbf{v}''(t), \mathbf{v}(t) \rangle_2\rvert \leq \|\mathbf{v}''(t)\|_2$. We now define
\begin{equation*}
\phi^\ast = \inf_{t \in [a, b]} \phi(t), \quad \phi(t) = \arcsin(\|\mathbf{v}''(t)\|_2^{-1}).
\end{equation*}
Obviously, $\phi^\ast \in (0, \pi/2]$. In the following, we will assume that
\begin{equation}\label{eq1}
\|\mathbf{v}(t_1)- \mathbf{v}(t_2) \|_2  \leq \frac{1-\cos \phi^\ast}{2},  \quad \| \mathbf{v}^{\prime}(t_1) - \mathbf{v}'(t_2)\|_2 \leq \frac{1}{4}-\frac{(1+\cos\phi^\ast)^2}{16}
\end{equation}
for all $t_1, t_2 \in [a, b]$. The latter can of course always be satisfied by choosing the interval $[a, b]$ small enough. If $\delta \in (0, 1]$ is sufficiently small, we write $[a, b]_\delta = [a+\delta, b-\delta]$ and define
\begin{equation}\label{d_delta}
d_\delta = \min \! \left(\frac{\delta}{5} \sin(\phi^\ast/2), \, \inf_{\substack{t \in[a, b]_\delta \\ s\in [0, L]\setminus (a, b) }} d( \mathbf{v}(t) ,\mathbf{v}(s)) \right). 
\end{equation}
Note that $d_\delta >0$ since the continuous map $(t, s)\mapsto d(\mathbf{v}(t) , \mathbf{v}(s))$ attains its minimum on the compact set $[a, b]_\delta \times [0, L]\setminus (a, b)$ and this value must be greater than zero due to the injectivity of $\mathbf{v}$. From now on we assume $\delta$ to be fixed. Our results will be stated in the form of asymptotic estimates for the inner products $ \langle \mathbf{1}_A, \mathcal{D}(\mathbf{x}, \mathbf{r} ) \Psi_{\scriptscriptstyle K}^{\scriptscriptstyle N} \rangle$ which hold uniformly for all $(\mathbf{x}, \mathbf{r}) \in UT\mathbb{S}^2$ with
\begin{equation}\label{neighborhood}
d(\mathbf{x}, \partial A) = d(\mathbf{x}, \mathbf{v}([a, b]_\delta) ), \quad d(\mathbf{x}, \partial A) \leq \frac{\phi^\ast}{4}.
\end{equation}
As illustrated in \autoref{Fig3}, the set of all points $\mathbf{x}\in \mathbb{S}^2$ satisfying \eqref{neighborhood} constitutes a certain neighborhood of the curve segment $ \mathbf{v}([a, b]_\delta) $ and will be denoted by $\mathcal{N}(\mathbf{v}([a, b]_\delta))$. In particular, we point out that for each $\mathbf{x} \in \mathcal{N}(\mathbf{v}([a, b]_\delta))$ there exists a unique $p_\mathbf{x}\in [a, b]_\delta$ such that
\begin{equation}\label{eq3}
d(\mathbf{x}, \partial A) = d(\mathbf{x}, \mathbf{v}(p_\mathbf{x})).
\end{equation}
Indeed, the existence of such a point is obvious. To verify the uniqueness, let us assume that $p$ and $q$ are two different points in $[a, b]_\delta$ with $d(\mathbf{x}, \partial A)=d(\mathbf{x},\mathbf{v}(p)) = d(\mathbf{x}, \mathbf{v}(q)) $. Then, the spherical cap which is centered at $\mathbf{x}$ and has the opening angle $d(\mathbf{x}, \partial A)\leq \phi^\ast /4$, the upper bound following from the condition in \eqref{neighborhood}, clearly does not contain any point of $\partial A$ since we defined spherical caps to be open sets. Consequently, as the boundary of the mentioned cap has a curvature of at least $1/\sin(\phi^\ast/4)$, there must exist a point $o$ between $p$ and $q$ such that $\|\mathbf{v}''(o) \|_2^{-1} \leq \sin(\phi^{\ast}/4)$. This, however, contradicts the definition of $\phi^\ast$. In the following, when $\mathbf{x}\in  \mathcal{N}(\mathbf{v}([a, b]_\delta))$ is given, we will always use the symbol $p_\mathbf{x}$ to refer to the unique point in $[a, b]_\delta$ satisfying \eqref{eq3}.
\begin{figure}[t]
\centering
\includegraphics[width=0.42\textwidth]{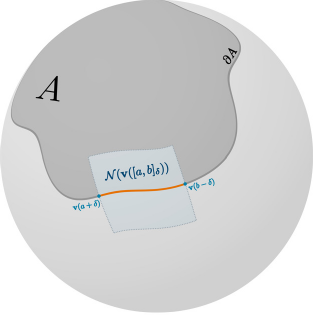}
\caption{The set $\mathcal{N}(\mathbf{v}([a, b]_\delta))$}\label{Fig3}
\end{figure}

\subsection{A key lemma}
Before presenting our main results, we will prove the following auxiliary lemma, which gives an upper bound on some approximation error when representing the smooth curve segment locally by circles. This will enable us to reduce the problem of detecting jump discontinuities along $C^3$-curves to estimating the frame coefficients of indicator functions of spherical caps. We note that a similar approach was taken in \citep{bib5} for the Euclidean setting.
\begin{lemma}\label{lemma1}
Let $p \in [a, b]_\delta$. Then there exists a unique spherical cap $C_p$ whose boundary  $\partial C_p$ has an arc-length parameterization 
\begin{equation*}
\mathbf{u}(t), \quad t \in [-\pi \|\mathbf{v}''(p) \|_2^{-1}, \pi \|\mathbf{v}''(p) \|_2^{-1}],
\end{equation*}
i.e., the length of $\partial C_p$ equals $ 2\pi\|\mathbf{v}''(p) \|_2^{-1}$, such that
\begin{enumerate}
\item $\mathbf{u}(0) = \mathbf{v}(p)$, $\mathbf{u}'(0) = \mathbf{v}'(p)$, $\mathbf{u}''(0) = \mathbf{v}''(p)$ and
\item for each $\varepsilon>0$ there exists some $\mathbf{a}\in A\cap C_p$ with $\mathbf{a}\perp \mathbf{v}'(p) $ and $d(\mathbf{a}, \mathbf{v}(p))<\varepsilon$, i.e., the boundary curve $\mathbf{u}$ is positively oriented.
\end{enumerate}
Furthermore, if $\mathbf{x} \in \mathcal{N}(\mathbf{v}([a, b]_\delta))$ and $p_\mathbf{x}\in [a, b]_\delta$ is the corresponding point with $d(\mathbf{x}, \partial A) = d(\mathbf{x}, \mathbf{v}(p_\mathbf{x}))$, then
\begin{equation}\label{eq9}
\int_{C(\mathbf{x}, r)} \lvert \mathbf{1}_{A}(\mathbf{y}) - \mathbf{1}_{C_{p_\mathbf{x}}}(\mathbf{y}) \rvert \, \mathrm{d}\omega(\mathbf{y}) \leq c\left( \sup_{t \in [a, b]} \| \mathbf{v}'''(t) \|_2 + 1 \right)  r^{4}, \quad   r \in [0, d_\delta/2).
\end{equation}
Here, $c=c(\phi^\ast)>0$ is a global constant which does not depend on any of the parameters. In particular, the upper bound does not depend on the choice of $\mathbf{x}$.
\end{lemma}
In the proof of this lemma, we will make use of the following relation between the distance of two points with equal latitude and the difference of their longitudinal coordinates.
\begin{lemma}\label{lemma3}
Let $\varphi \in [-\pi, \pi]$ and $\phi \in [0, \pi]$. Then
\begin{equation*}
d(\mathbf{x}(\phi, 0), \mathbf{x}(\phi, \varphi) ) = \arccos(1+(\cos \varphi -1)\sin^2\phi)\geq \sqrt{1- \frac{\varphi^2}{12}} \, \lvert \varphi \rvert \sin\phi.
\end{equation*}
In particular, for $\phi \in (0, \pi)$, $\varphi \mapsto d(\mathbf{x}(\phi, 0), \mathbf{x}(\phi, \varphi) )$ is an even function which is strictly increasing on $[0, \pi]$ and
\begin{equation*}
\lvert \varphi \rvert \leq \frac{5 \, d(\mathbf{x}(\phi, 0), \mathbf{x}(\phi, \varphi) )}{2\sin \phi}.
\end{equation*}
\end{lemma}
\begin{proof}
Since the first statement is trivial for $\phi\in\{ 0, \pi\}$, we can assume that $\phi\in (0, \phi)$. Using \eqref{polarkoord.}, the equality
\begin{equation*}
d(\mathbf{x}(\phi, 0), \mathbf{x}(\phi, \varphi) ) = \arccos(1+(\cos \varphi -1)\sin^2\phi)
\end{equation*}
follows by a straightforward calculation. Additionally, the inequality $\arccos z$ $\geq \sqrt{2 (1-z)}$ yields 
\begin{equation*}
\arccos(1+(\cos \varphi -1) \sin^2\phi)\geq  \sqrt{2 (1-\cos \varphi)  \sin^2\phi}.
\end{equation*}
Finally, by applying $\cos \varphi \leq 1-\varphi^2/2 + \varphi^4/24$, we obtain the desired inequality.
\end{proof}
\begin{proof}[Proof of {\hyperref[lemma1]{Lemma~\ref*{lemma1}}}]
Let $p \in [a, b]_\delta$. Since the curvature $\|\mathbf{v}''(p)\|_2$ of $\mathbf{v}$ at $p$ must be at least $1$, we can choose $\phi \in (0, \pi/2]$ such that $\sin \phi = \|\mathbf{v}''(p) \|_2^{-1}$.
Without loss of generality, let
\begin{equation*}
\mathbf{v}(p) = \begin{pmatrix}
\sin \phi \\
0 \\
\cos \phi
\end{pmatrix}, \quad \mathbf{v}'(p) \in \{ \mathbf{e}_2, -\mathbf{e}_2\}.
\end{equation*}
It then follows that either
\begin{equation*}
\mathbf{v}''(p) = -\frac{1}{\sin \phi}\mathbf{e}_1  \quad \text{or} \quad \mathbf{v}''(p) = -\frac{1}{\sin \phi} \begin{pmatrix}
- \cos 2\phi \\
0 \\
\sin 2\phi
\end{pmatrix}.
\end{equation*}
Indeed, differentiating both sides of $\langle \mathbf{v}'(t), \mathbf{v}'(t) \rangle_2 \equiv 1$ yields $\mathbf{v}''(p) \perp \mathbf{e}_2$. Similarly, by differentiating both sides of $\langle \mathbf{v}(t), \mathbf{v}(t)\rangle_2 \equiv 1$ twice, we see that $\langle \mathbf{v}''(p), \mathbf{v}(p) \rangle_2 = -1$. Thus, if $\mathbf{v}''(p) \sin \phi = (\cos z, 0, \sin z)$, $z \in [0, 2\pi)$, we have
\begin{equation*}
 -\sin \phi =  \langle \mathbf{v}''(p) \sin \phi, \mathbf{v}(p) \rangle_2 =\cos z \sin \phi + \sin z \cos \phi =\sin(z+\phi)
\end{equation*}
and consequently $z=\pi$ or $z=2\pi-2\phi$. Since rotating the whole setup by $\mathbf{R}_{\mathbf{e}_3}(\pi)\mathbf{R}_{\mathbf{e}_2}(-2\phi)$ transfers $(-\cos 2\phi, 0, \sin 2\phi)$ to $\mathbf{e}_1$, switches the sign of $\mathbf{v}'(p)$ and keeps $\mathbf{v}(p)$ invariant, we can assume that $\mathbf{v}''(p)  = -(\sin \phi)^{-1}\mathbf{e}_1$. Now, if $\mathbf{v}'(p) = \mathbf{e}_2$, then
\begin{equation}\label{eq2}
\mathbf{u}(t) = \begin{pmatrix}
\sin \phi \cos (t/\sin \phi) \\
\sin \phi \sin (t/\sin \phi) \\
\cos \phi
\end{pmatrix}, \quad t \in [-\pi \sin \phi, \pi \sin \phi],
\end{equation}
defines a unit speed curve with $\mathbf{u}(0) = \mathbf{v}(p)$, $\mathbf{u}'(0) = \mathbf{v}'(p)$ and $\mathbf{u}''(0) = \mathbf{v}''(p)$, as one can easily verify. For $\mathbf{v}'(p) = -\mathbf{e}_2$, choose the inverted curve $t \mapsto \mathbf{u}(-t)$. Indeed, the image of $\mathbf{u}$ yields the unique circle on the sphere which has a parameterization satisfying these properties. Furthermore, there are exactly two spherical caps whose boundary is equal to the image of $\mathbf{u}$, only one of which, we will call it $C_p$, satisfying the second condition stated in \hyperref[lemma1]{Lemma~\ref*{lemma1}}. This completes the first part of the proof.

Now, let $\mathbf{x} \in \mathcal{N}(\mathbf{v}([a, b]_\delta))$ and $p_\mathbf{x}\in [a, b]_\delta$ with $d(\mathbf{x}, \partial A) = d(\mathbf{x}, \mathbf{v}(p_\mathbf{x}))$. Also, let $\phi \in (0, \pi/2]$ such that $\sin \phi = \| \mathbf{v}''(p_\mathbf{x})\|_2^{-1}$. Just like in the first part of the proof, we can assume that $\mathbf{v}(p_\mathbf{x}) = (\sin \phi, 0, \cos \phi)$, $\mathbf{v}'(p_\mathbf{x})\in \{ \mathbf{e}_2, -\mathbf{e}_2\}$ and $\mathbf{v}''(p_\mathbf{x}) = -(\sin\phi)^{-1}\mathbf{e}_1$. Moreover, we only consider the case $\mathbf{v}'(p_\mathbf{x})=\mathbf{e}_2$ as otherwise the proof is completely analogous. In the following, we will use the notation
\begin{equation}\label{eq5}
\mathbf{v}(t) = \mathbf{v}( \theta_\mathbf{v} (t), \varphi_\mathbf{v} (t)) = \begin{pmatrix}
\sin \theta_\mathbf{v} (t) \, \cos \varphi_\mathbf{v} (t) \\
\sin \theta_\mathbf{v} (t) \, \sin \varphi_\mathbf{v} (t) \\
\cos \theta_\mathbf{v} (t)
\end{pmatrix} = \begin{pmatrix}
v_1(t) \\
v_2(t) \\
v_3(t)
\end{pmatrix}, \quad t \in [0, L],
\end{equation}
which is justified by \eqref{polarkoord.}. Clearly, $ \varphi_\mathbf{v} (p_\mathbf{x}) = 0$ and $\theta_\mathbf{v} (p_\mathbf{x}) = \phi$. Additionally, \eqref{eq1} implies that for each $t \in [a, b]$ we have
\begin{align*}
\|\mathbf{v}(t)- \mathbf{v}(p_\mathbf{x}) \|_2 \leq \frac{1-\cos \phi}{2}, \quad \| \mathbf{v}^{\prime}(t) - \mathbf{v}'(p_\mathbf{x})\|_2 \leq\frac{4-(1+\cos\phi)^2}{16},
\end{align*}
and therefore
\begin{equation}\label{eq18}
\lvert v_2(t)\rvert \leq (1-\cos \phi)/2, \quad \lvert v_3(t) \rvert \leq (1+\cos \phi)/2,
\end{equation}
as well as 
\begin{equation}\label{eq4}
v_2'(t) \in [3/4 +(1+\cos\phi)^2/16 , 1], \quad \lvert v_3'(t)\rvert \leq 1/4-(1+\cos\phi)^2/16.
\end{equation}
Using \eqref{eq18}, an elementary calculation yields
\begin{equation*}
\left\lvert \frac{v_2(t)}{\sqrt{1-v_3(t)^2}} \right\rvert \leq \frac{1-\cos \phi}{2\sqrt{1-(1+\cos \phi)^2/4}} = \frac{\sqrt{1-\cos \phi}}{\sqrt{3+\cos \phi}} \leq \frac{1}{\sqrt{3}}, \quad t \in [a, b].
\end{equation*}
According to \eqref{eq5}, we have
\begin{equation*}
\varphi_\mathbf{v}(t) = \arcsin\!\left( \frac{v_2(t)}{\sin(\arccos(v_3(t)))} \right) = \arcsin\!\left( \frac{v_2(t)}{\sqrt{1-v_3(t)^2}} \right) , \quad t \in [a, b].
\end{equation*}
Since $\arcsin$ is a function in $C^3([-1/\sqrt{3}, 1/\sqrt{3}])$, we conclude that $\varphi_\mathbf{v} \in C^3([a, b])$. Furthermore,
\begin{equation}\label{eq6}
\varphi_\mathbf{v}'(t) = \arcsin'\!\left( \frac{v_2(t)}{\sqrt{1-v_3(t)^2}} \right) \left( \frac{v_2'(t)}{\sqrt{1-v_3(t)^2}}  + \frac{v_2(t) v_3(t) v_3'(t)}{(1-v_3(t)^2)^{3/2}} \right).
\end{equation}
Here, 
\begin{equation}\label{eq19}
\left( \frac{v_2'(t)}{\sqrt{1-v_3(t)^2}}  + \frac{v_2(t) v_3(t) v_3'(t)}{(1-v_3(t)^2)^{3/2}} \right) \geq 1/2,
\end{equation}
which can be verified as follows. Using \eqref{eq18} and \eqref{eq4}, together with the fact that $\lvert v_2(t)\rvert \leq \sqrt{1-v_3(t)^2}$, we first obtain 
\begin{equation*}
\left( \frac{v_2'(t)}{\sqrt{1-v_3(t)^2}}  + \frac{v_2(t) v_3(t) v_3'(t)}{(1-v_3(t)^2)^{3/2}} \right) \geq\frac{3}{4} - \frac{\lvert v_3(t) v_3'(t) \rvert}{1-v_3(t)^2}.
\end{equation*}
Additionally, by \eqref{eq4}, it holds that
\begin{equation*}
\lvert v_3'(t)\rvert \leq \frac{1-\left(\frac{1}{2} (1+\cos\phi)\right)^2}{4}\leq \frac{1-v_3(t)^2}{4 \lvert v_3(t) \rvert},
\end{equation*}
which confirms \eqref{eq19}. Therefore, together with
\begin{equation*}
\inf_{x \in (-1,1)} \arcsin'(x) = \inf_{x \in (-1,1)} \frac{1}{\sqrt{1-x^2}} = 1,
\end{equation*}
we obtain
\begin{equation}\label{eq10}
\varphi_\mathbf{v}'(t) \geq 1/2, \quad t \in [a, b].
\end{equation}
In particular, $\varphi_\mathbf{v}$ is strictly increasing on $[a, b]$ and thus the inverse function $t_\mathbf{v} \colon [\varphi_\mathbf{v}(a), \varphi_\mathbf{v}(b)] \rightarrow [a, b]$, where
\begin{equation*}
t_\mathbf{v}(\varphi_\mathbf{v}(t))) = t \quad \text{for all } t \in [a, b],
\end{equation*}
exists. Of course, $t_\mathbf{v}$ is continuous. Additionally, $t_\mathbf{v} \in C^1( (\varphi_\mathbf{v}(a), \varphi_\mathbf{v}(b)) )$ with
\begin{equation*}
t_\mathbf{v}'(\varphi) = \frac{1}{\varphi_\mathbf{v}'(t_\mathbf{v}(\varphi))}\leq 2.
\end{equation*}
Consequently, the mean value theorem yields
\begin{equation}\label{eq12}
\lvert t_\mathbf{v}(\varphi_1)  -  t_\mathbf{v}(\varphi_2)\rvert \leq 2 \lvert \varphi_1 - \varphi_2 \rvert  \quad \text{for all } \varphi_1, \varphi_2 \in [\varphi_\mathbf{v}(a), \varphi_\mathbf{v}(b)].
\end{equation}

After these preliminary considerations, we can now approximate $\partial A$ locally, meaning at the point $\mathbf{v}(p_\mathbf{x})$, by the boundary of a suitable spherical cap. Note that by Taylor's theorem we have
\begin{equation}\label{eq7}
\mathbf{v}(t) = \mathbf{v}(p_\mathbf{x}) + \mathbf{v}'(p_\mathbf{x})(t-p_\mathbf{x}) +\frac{1}{2} \mathbf{v}''(p_\mathbf{x}) (t-p_\mathbf{x})^2+ \frac{1}{6}\! \begin{pmatrix}
v_1'''(\xi_1) \\
v_2'''(\xi_2) \\
v_3'''(\xi_3)
\end{pmatrix} (t-p_\mathbf{x})^3
\end{equation}
for each $t \in [a, b]$, where $\xi_i$ lies between $t$ and $p_\mathbf{x}$ and 
\begin{equation*}
\left \| \begin{pmatrix}
v_1'''(\xi_1) \\
v_2'''(\xi_2) \\
v_3'''(\xi_3)
\end{pmatrix}  \right \|_2 \leq \sqrt{3} \sup_{t \in [a, b]} \| \mathbf{v}'''(t) \|_2.
\end{equation*}
We will now consider the unique spherical cap $C_{p_\mathbf{x}}$ given in the first part of the proof. Locally, its boundary corresponds to the curve
\begin{equation*}
\tilde{\mathbf{u}}(t) = \begin{pmatrix}
\sin \phi \cos (\varphi_\mathbf{v}(t) ) \\
\sin \phi \sin (\varphi_\mathbf{v}(t) ) \\
\cos \phi
\end{pmatrix}, \quad t \in [a, b].
\end{equation*}
This re-parameterization of \eqref{eq2} allows us express the boundary $\partial C_{p_\mathbf{x}}$ locally in such a way that the longitudinal coordinates of $\tilde{\mathbf{u}}(t)$ and $\mathbf{v}(t)$ are equal for all $t \in [a, b]$. Using \eqref{eq6}, it is straightforward to show that $\varphi_\mathbf{v}'(p_\mathbf{x}) = (\sin \phi)^{-1}$, $\varphi_\mathbf{v}''(p_\mathbf{x}) = 0$ as well as
\begin{equation*}
\tilde{\mathbf{u}}(p_\mathbf{x}) = \mathbf{v}(p_\mathbf{x}), \quad \tilde{\mathbf{u}}'(p_\mathbf{x}) = \mathbf{v}'(p_\mathbf{x}), \quad \tilde{\mathbf{u}}''(p_\mathbf{x}) = \mathbf{v}''(p_\mathbf{x}).
\end{equation*}
Consequently, Taylor's theorem yields
\begin{equation}\label{eq20}
\tilde{\mathbf{u}}(t) = \mathbf{v}(p_\mathbf{x}) + \mathbf{v}'(p_\mathbf{x})(t-p_\mathbf{x}) + \frac{1}{2}\mathbf{v''}(p_\mathbf{x}) (t-p_\mathbf{x})^2+ \frac{1}{6}\! \begin{pmatrix}
\tilde{u}_1'''(\xi_1) \\
\tilde{u}_2'''(\xi_2) \\
\tilde{u}_3'''(\xi_3)
\end{pmatrix} (t-p_\mathbf{x})^3
\end{equation}
for $t \in [a, b]$, where $\xi_i$ lies between $t$ and $p_\mathbf{x}$. Additionally,  elementary calculations yield
\begin{equation*}
\left \| \begin{pmatrix}
\tilde{u}_1'''(\xi_1) \\
\tilde{u}_2'''(\xi_2) \\
\tilde{u}_3'''(\xi_3)
\end{pmatrix}  \right \| \leq c_{\phi^\ast} \left( \sup_{t \in [a, b]} \| \mathbf{v}'''(t) \|_2 + 1\right).
\end{equation*}
Thus, combining \eqref{eq7} and \eqref{eq20}, we obtain
\begin{equation}\label{eq11}
\| \mathbf{v}(t) - \tilde{\mathbf{u}}(t) \|_2 \leq c_{\phi^\ast} \left( \sup_{t \in [a, b]} \| \mathbf{v}'''(t) \|_2 +1 \right) \lvert t-p_\mathbf{x}\rvert^3, \quad t \in [a, b].
\end{equation}

Finally, let us prove the estimate \eqref{eq9}. Under the given assumptions, $C(\mathbf{x}, r)$ does not contain any points $\mathbf{v}(s)$ with $s \in [0, L]\setminus [a, b]$. Indeed, if such a point $\mathbf{v}(s) \in C(\mathbf{x},r)$ would exist, then $d(\mathbf{v}(s), \mathbf{x}) \leq r< d_\delta /2$. But since $\mathbf{v}(p_\mathbf{x})$ has the minimal distance towards $\mathbf{x}$ among all points of $\partial A$, we conclude that $d( \mathbf{v}(p_\mathbf{x}), \mathbf{x} ) < d_\delta /2$ and therefore $d(\mathbf{v}(p_\mathbf{x}),\mathbf{v}(s) )  < d_\delta $. However, this contradicts the definition of $d_\delta$.

Next, we show that for each $\mathbf{y}(\theta, \varphi) \in C(\mathbf{x}, r)$ it holds that $\lvert \varphi \rvert \leq c \,r$ and, in particular, $\varphi \in [\varphi_\mathbf{v}(a), \varphi_\mathbf{v}(b)]$. First, it should be noted that $\mathbf{x} = \mathbf{x}(\theta, 0)$ with $\lvert \theta-\phi \rvert \leq \phi^\ast/4$. Indeed, for $\mathbf{x} = \mathbf{v}(p_\mathbf{x})$ this holds trivially and for $\mathbf{x} \neq \mathbf{v}(p_\mathbf{x})$ the function $ t \mapsto d(\mathbf{v}(t), \mathbf{\mathbf{x}})$ is continuously differentiable in a neighborhood of $t=p_\mathbf{x}$ and attains a global minimum there. Therefore,
\begin{equation*}
\frac{\mathrm{d}}{\mathrm{d} t}  d(\mathbf{v}(t), \mathbf{x}) \big\vert_{t=p_\mathbf{x}} = \arccos'( \langle \mathbf{v}(p_\mathbf{x}), \mathbf{x}\rangle_2) \, \langle \mathbf{v}'(p_\mathbf{x}),  \mathbf{x}\rangle_2 =0,
\end{equation*}
which implies that $\mathbf{x}\perp \mathbf{e}_2$. Thus, taking into account that $d(\mathbf{x}, \mathbf{v}(p_\mathbf{x}))\leq \phi^\ast/4$ by \eqref{eq3}, we have $\mathbf{x} = \mathbf{x}(\theta, 0)$ with $\lvert \theta-\phi \rvert \leq \phi^\ast/4$, as claimed. Consequently, since $ \phi^\ast \leq \phi \leq \pi/2$ and $r \leq \phi^\ast/4$, it is now obvious that the absolute value of the longitudinal coordinate $\varphi$ of any point $\mathbf{y}(\theta, \varphi) \in C(\mathbf{x}, r)$ can not exceed $\lvert \tilde{\varphi} \rvert$ with $d(\mathbf{x}(\phi^\ast/2, \tilde{\varphi}), \mathbf{x}(\phi^\ast/2,0)) = r$. By \hyperref[lemma3]{Lemma~\ref*{lemma3}} and \eqref{d_delta}, we obtain $\lvert \tilde{\varphi} \rvert \leq 5 \, r/(2\sin(\phi^\ast /2))\leq \delta/2$, where the last inequality follows from $r\leq d_\delta /2$ and \eqref{d_delta}. Finally, we note that by the mean value theorem and \eqref{eq10} we have
\begin{equation*}
\varphi_\mathbf{v}(a+\delta) -  \varphi_\mathbf{v}(a), \varphi_\mathbf{v}(b) -  \varphi_\mathbf{v}(b-\delta) \geq \delta/2
\end{equation*}
and therefore  $\lvert \varphi \rvert \leq \lvert \tilde{\varphi} \rvert \leq \min(-\varphi_\mathbf{v}(a), \varphi_\mathbf{v}(b))$.

With these considerations in mind, we have
\begin{equation}\label{eq14}
\int_{C(\mathbf{x},r)} \lvert \mathbf{1}_{A}(\mathbf{y}) - \mathbf{1}_{C_{p_\mathbf{x}}}(\mathbf{y}) \rvert \, \mathrm{d}\omega(\mathbf{y}) \leq \int_{-cr}^{cr} d( \mathbf{v}(t_\mathbf{v}(\varphi)) ,\tilde{\mathbf{u}}(t_\mathbf{v}(\varphi))) \, \mathrm{d}\varphi.
\end{equation}
Furthermore, by \eqref{normeq} it holds that
\begin{equation*}
d( \mathbf{v}(t_\mathbf{v}(\varphi)) , \tilde{\mathbf{u}}(t_\mathbf{v}(\varphi))) \leq \frac{\pi}{2} \|\mathbf{v}(t_\mathbf{v}(\varphi)) - \tilde{\mathbf{u}}(t_\mathbf{v}(\varphi)) \|_2.
\end{equation*}
Hence, using \eqref{eq11} and \eqref{eq12}, we obtain
\begin{equation*}
d( \mathbf{v}(t_\mathbf{v}(\varphi)) ,\tilde{\mathbf{u}}(t_\mathbf{v}(\varphi))) \leq c_{\phi^\ast} \left( \sup_{t \in [a, b]} \| \mathbf{v}'''(t) \|_2 +1 \right) \lvert \varphi\rvert^3, \quad \lvert \varphi \rvert < \delta/2. 
\end{equation*}
Lastly, combining this estimate with \eqref{eq14} yields the desired bound in \eqref{eq9}.
\end{proof}

\subsection{Main results}
We will now state our main theorems, which reveal clear connections between geometrical properties of $\partial A$ and the frame coefficients with respect to $\mathbf{1}_A$. As illustrated in \autoref{Fig4}, it is clear that the inner products $\langle \mathbf{1}_A, \mathcal{D}(\mathbf{x}, \mathbf{r} ) \Psi_{\scriptscriptstyle K}^{\scriptscriptstyle N} \rangle$ will depend heavily on both the position and orientation of the wavelet compared to the closest edge segment. 

\begin{figure}[t]
\centering
\includegraphics[angle = 0, width=0.65\textwidth]{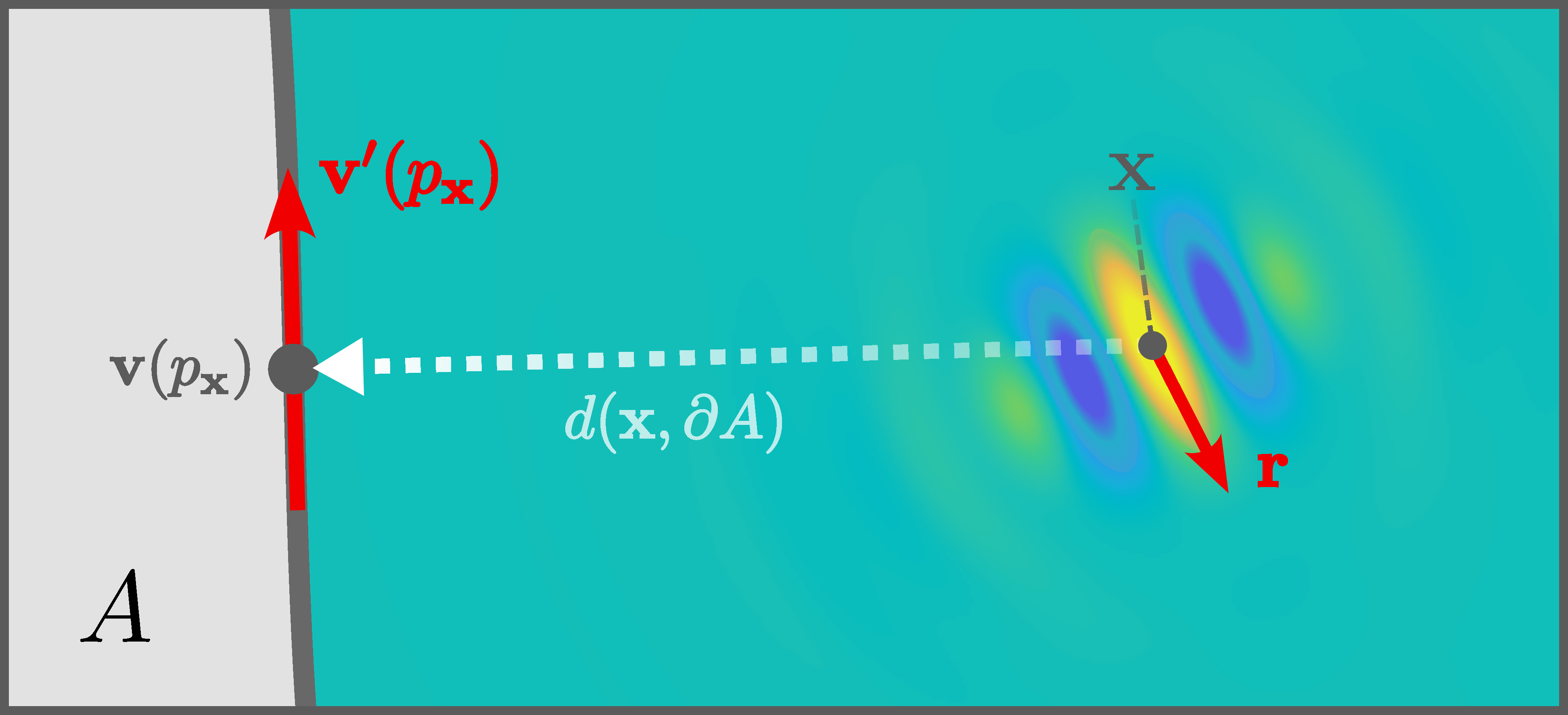}
\caption{A localized wavelet $\mathcal{D}(\mathbf{x}, \mathbf{r}) \Psi_{ \scriptscriptstyle K}^{\scriptscriptstyle N}$ in	the vicinity of the boundary $\partial A$}\label{Fig4}
\end{figure}

\begin{theorem}\label{thm1}
Let $\Psi_{ \scriptscriptstyle K}^{\scriptscriptstyle N}$ be the directional wavelet defined in \eqref{directional wavelets} with $\kappa \in C^{3q-1}([0, \infty))$, or $\kappa \in C^q([0, \infty))$ if $K=1$, where $q= K+\max(5, 1+\lceil \ell(\ell+4)/2\rceil)$ and $\ell \in \mathbb{N}$. Then there exist constants $c_1$, $c_2$, $N_0>0$ as well as some nonempty open interval $I \subset (0, \infty)$, such that for all $(\mathbf{x}, \mathbf{r})\in UT\mathbb{S}^2$ with $\mathbf{x}\in  \mathcal{N}(\mathbf{v}([a, b]_\delta))$ it holds that
\begin{equation}\label{wavelets asympt2}
\lvert \langle \mathbf{1}_A, \mathcal{D}(\mathbf{x}, \mathbf{r} ) \Psi_{\scriptscriptstyle K}^{\scriptscriptstyle N} \rangle \rvert \geq c_1 \,  \lvert \chi_{\scriptscriptstyle K}(d_\mathbf{x}(\mathbf{r}, \mathbf{v}'(p_\mathbf{x}))) \rvert \quad \text{if } N d(\mathbf{x}, \partial A) \in I, \; N \geq N_0,
\end{equation}
and
\begin{equation}\label{wavelets asympt3}
\lvert \langle \mathbf{1}_A, \mathcal{D}(\mathbf{x}, \mathbf{r}) \Psi_{ \scriptscriptstyle K}^{\scriptscriptstyle N} \rangle \rvert \leq \frac{c_2 \left( \sup_{t \in [a, b]} \| \mathbf{v}'''(t) \|_2 +1 \right)  }{(1+N d( \mathbf{x}, \partial A))^\ell} \quad \text{if } N>(2/d_\delta)^{(1-1/(\ell+3))^{-1}}.
\end{equation}
The dependencies of the constants are $c_1=c_1(\kappa, \phi^\ast, K)$, $c_2 = c_2(\kappa, \phi^\ast, K, \ell)$, $N_0 = N_0(\kappa, \phi^\ast, K, d_\delta, \sup_{t \in [a, b]} \| \mathbf{v}'''(t) \|_2)$ and $I = I(\kappa, \phi^\ast, K)$.
\end{theorem}
\noindent Roughly speaking, \autoref{thm1} states that the coefficients $\langle \mathbf{1}_A,  \mathcal{D}(\mathbf{x}, \mathbf{r}) \Psi_{ \scriptscriptstyle K}^{\scriptscriptstyle N} \rangle$ peak when both the position $\mathbf{x}$ and orientation $\mathbf{r}$ of the wavelet somewhat match the position $\mathbf{v}(p_\mathbf{x})$ and orientation $\mathbf{v}'(p_\mathbf{x})$ of the nearest edge. Furthermore, these peaks remain stable in magnitude and move closer towards the boundary $\partial A$ as $N$ increases. Away from the edge, the inner products decay rapidly.

The following theorem provides us with more insight on the structure of the frame coefficients in the vicinity of a singularity. Here, for $\mathbf{x}\in \mathbb{S}^2$, we make use of the notation
\begin{equation*}
(-1)^{\mathbf{x}\notin A}= \begin{cases}
1, \quad &\text{if }\mathbf{x}\in A,\\
-1,  &\text{if }\mathbf{x}\notin A, 
\end{cases}
\end{equation*}
as well as $(-1)^{\mathbf{x}\in A} = -(-1)^{\mathbf{x}\notin A}$.
\begin{theorem}\label{thm2}
Let $\Psi_{ \scriptscriptstyle K}^{\scriptscriptstyle N}$ be the directional wavelet defined in \eqref{directional wavelets} with $\kappa \in C^{3q-1}([0, \infty))$, or $\kappa \in C^q([0, \infty))$ if $K=1$, where $q= K+5$. Then
\begin{align}\label{wavelets asympt1}
&\langle \mathbf{1}_A,  \mathcal{D}(\mathbf{x}, \mathbf{r}) \Psi_{ \scriptscriptstyle K}^{\scriptscriptstyle N} \rangle  = \sqrt{\frac{2^{-1}\pi^{-3}\| \mathbf{v}''(p_\mathbf{x})\|_2^{-1}}{ \sin((-1)^{\mathbf{x}\in A}d(\mathbf{x}, \partial A) + \arcsin(\| \mathbf{v}''(p_\mathbf{x})\|_2^{-1}))}} \nonumber \\
& \quad\quad \times \chi_{\scriptscriptstyle K}(d_\mathbf{x}(\mathbf{r}, \mathbf{v}'(p_\mathbf{x}))) \, ((-1)^{\mathbf{x}\notin A})^K  \nonumber\\
&  \quad\quad   \times \int_{1/2}^{2 } \kappa(t)\,  t^{-1}\,  \cos\!\left( N d(\mathbf{x}, \partial A) \cdot t + \frac{ 2 d(\mathbf{x}, \partial A) - \pi(1-(-1)^K)}{4}\right)\mathrm{d} t \nonumber \\
& \quad \quad  + R_{ \scriptscriptstyle K}^{\scriptscriptstyle N}(\mathbf{x}, \mathbf{r}),
\end{align}
where
\begin{equation*}
\lvert R_{ \scriptscriptstyle K}^{\scriptscriptstyle N}(\mathbf{x}, \mathbf{r})\rvert \leq c_0 \, N^{-1},
\end{equation*}
holds for every $(\mathbf{x}, \mathbf{r})\in UT\mathbb{S}^2$ with $\mathbf{x}\in  \mathcal{N}(\mathbf{v}([a, b]_\delta))$. Here, we have $c_0 = c_0(\kappa, \phi^\ast, K, d_\delta, \sup_{t \in [a, b]}\|\mathbf{v}'''(t)\|_2)$.
\end{theorem}

\autoref{thm2} shows how the local geometry of the boundary $\partial A$ is reflected in the frame coefficients at high frequencies, that is, when $N$ becomes large. Most notably, the orientational dependency of the dominating term given by \eqref{wavelets asympt1} is completely characterized by the directionality function $\chi_{\scriptscriptstyle K}$.  Also, the parameter $N$ only appears in the integral expression, which can be viewed as a function in $d(\mathbf{x}, \partial A)$ that essentially gets dilated as $N$ increases.

\begin{remark}
We note that the results presented above can be applied to all functions which are (locally) of the form $g + \mathbf{1}_A$, where $g \in C(\mathbb{S}^2)$. More precisely,
\begin{equation*}
\lim_{N \rightarrow \infty} \sup_{(\mathbf{x}, \mathbf{r})\in UT\mathbb{S}^2}  \, \lvert \langle g, \mathcal{D}(\mathbf{x}, \mathbf{r}) \Psi_{ \scriptscriptstyle K}^{\scriptscriptstyle N} \rangle\rvert = 0
\end{equation*}
and therefore the upper and lower bounds \eqref{wavelets asympt2} and \eqref{wavelets asympt3} show that jump discontinuities along smooth curves can still be detected in this case. For more details on this matter we refer to the discussion in \cite[Remark~1]{bib25}. 
\end{remark}

In order to prove \autoref{thm1} and \autoref{thm2}, we will need the following result, which deals with the special case when $A$ is a spherical cap.
\begin{theorem}\label{thm3}
Let $A=C(\mathbf{z}, \phi) $ with $\phi^\ast \leq \phi \leq \pi-\phi^\ast$ and $\Psi_{ \scriptscriptstyle K}^{\scriptscriptstyle N}$ be the directional wavelet defined in \eqref{directional wavelets} with $\kappa \in C^{q}([0, \infty))$ for some $q\in \mathbb{N}$. Then
\begin{align}\label{wavelets asympt4}
&\langle \mathbf{1}_A,  \mathcal{D}(\mathbf{x}, \mathbf{r}) \Psi_{ \scriptscriptstyle K}^{\scriptscriptstyle N} \rangle \nonumber \\
& \quad \quad  =  \sqrt{\frac{ \sin \phi}{2 \pi^3 \sin((-1)^{\mathbf{x}\in A}d(\mathbf{x}, \partial A) + \phi)}} \,\chi_{\scriptscriptstyle K}(d_\mathbf{x}(\mathbf{r}, \mathbf{v}'(p_\mathbf{x}))) \, ((-1)^{\mathbf{x}\notin A})^K \nonumber\nonumber \\
&\quad \quad  \quad \times \int_{1/2}^{2 } \kappa(t)\,  t^{-1}\,  \cos\!\left( N d(\mathbf{x}, \partial A) \cdot t + \frac{ 2d(\mathbf{x}, \partial A) - \pi(1-(-1)^K)}{4}\right)\mathrm{d} t \nonumber \\
& \quad \quad  \quad +R_{ \scriptscriptstyle K}^{\scriptscriptstyle N}(\mathbf{x}, \mathbf{r}),
\end{align}
where
\begin{equation*}
\lvert R_{ \scriptscriptstyle K}^{\scriptscriptstyle N}(\mathbf{x}, \mathbf{r})\rvert \leq c_0 \, N^{-1},
\end{equation*}
holds for all $(\mathbf{x}, \mathbf{r})\in UT\mathbb{S}^2$ with $d(\mathbf{x}, \partial A)\leq \phi^\ast/4$. Furthermore, there exist constants $c_1$, $c_2$, $N_0>0$ as well as some nonempty open interval $I \subset (0, \infty)$, such that for all $(\mathbf{x}, \mathbf{r})\in UT\mathbb{S}^2$ with $d(\mathbf{x}, \partial A)\leq \phi^\ast/4$ it holds that
\begin{equation*}
\lvert \langle \mathbf{1}_A, \mathcal{D}(\mathbf{x}, \mathbf{r} ) \Psi_{\scriptscriptstyle K}^{\scriptscriptstyle N} \rangle \rvert \geq c_1 \,  \lvert \chi_{\scriptscriptstyle K}(d_\mathbf{x}(\mathbf{r}, \mathbf{v}'(p_\mathbf{x}))) \rvert \quad \text{if } N d(\mathbf{x}, \partial A) \in I, \; N \geq N_0,
\end{equation*}
and
\begin{equation*}
\lvert \langle \mathbf{1}_A, \mathcal{D}(\mathbf{x}, \mathbf{r}) \Psi_{ \scriptscriptstyle K}^{\scriptscriptstyle N} \rangle \rvert \leq \frac{c_2 }{(1+N d( \mathbf{x}, \partial A))^q}.
\end{equation*}
The dependencies of the constants are $c_0 = c_0(\kappa, \phi^\ast, K)$, $c_1=c_1(\kappa, \phi^\ast, K)$, $c_2 = c_2(\kappa, \phi^\ast, K, q)$, $N_0 = N_0(\kappa, \phi^\ast, K)$ and $I = I(\kappa, K)$.
\end{theorem}
\begin{proof}
This statement was proven in \cite[Theorem~3.1]{bib25} in terms of Euler angles. An adaptation of the formulas with respect to the unit tangent bundle parameterization of the rotation operator can be achieved in a straightforward manner.
\end{proof}

Another important intermediate result, which we will use to prove our main theorems, is given by the following lemma.
\begin{lemma}\label{lemma4}
Let $\varepsilon \in [3/5,1)$ and $\Psi_{ \scriptscriptstyle K}^{\scriptscriptstyle N}$ be the directional wavelet defined in \eqref{directional wavelets} for $\kappa \in C^{3q-1}([0, \infty))$, or $\kappa \in C^q([0, \infty))$ if $K=1$, with $q\geq K-3 +2/(1-\varepsilon)$. Then there exists a constant $c=c(\kappa, \phi^\ast,K, \varepsilon)>0$ such that
\begin{equation}\label{eq15}
\lvert \langle \mathbf{1}_A - \mathbf{1}_{C_{p_\mathbf{x}}}, \mathcal{D}(\mathbf{x}, \mathbf{r}) \Psi_{\scriptscriptstyle K}^{\scriptscriptstyle N} \rangle \rvert \leq  c\left( \sup_{t \in [a, b]} \| \mathbf{v}'''(t) \|_2 +1 \right)  N^{2-4\varepsilon}
\end{equation}
holds for all $(\mathbf{x}, \mathbf{r}) \in UT\mathbb{S}^2$ with $ \mathbf{x} \in  \mathcal{N}(\mathbf{v}([a, b]_\delta))$, provided that $N^{-\varepsilon} < d_\delta/2$. Here, $C_{p_\mathbf{x}}$ is the corresponding spherical cap from \hyperref[lemma1]{Lemma~\ref*{lemma1}}.
\end{lemma}
\begin{proof}
Let $(\mathbf{x}, \mathbf{r}) \in UT\mathbb{S}^2$ with $ \mathbf{x} \in  \mathcal{N}(\mathbf{v}([a, b]_\delta)) $ and let $C_{p_\mathbf{x}}$ be the corresponding spherical cap from \hyperref[lemma1]{Lemma~\ref*{lemma1}}. We start with the simple estimate
\begin{align*}
\lvert \langle  \mathbf{1}_A -  \mathbf{1}_{C_{p_\mathbf{x}}},& \mathcal{D}(\mathbf{x}, \mathbf{r}) \Psi_{\scriptscriptstyle K}^{\scriptscriptstyle N} \rangle \rvert \\
&\leq \int_{C(\mathbf{x}, N^{-\varepsilon})} \lvert \mathcal{D}(\mathbf{x}, \mathbf{r}) \Psi_{\scriptscriptstyle K}^{\scriptscriptstyle N} (\mathbf{y}) \rvert \, \lvert \mathbf{1}_A(\mathbf{y}) - \mathbf{1}_{C_{p_\mathbf{x}}}(\mathbf{y}) \rvert \, \mathrm{d}\omega(\mathbf{y}) \\
& + \int_{\mathbb{S}^2\setminus C(\mathbf{x}, N^{-\varepsilon})} \lvert \mathcal{D}(\mathbf{x}, \mathbf{r}) \Psi_{\scriptscriptstyle K}^{\scriptscriptstyle N} (\mathbf{y}) \rvert \, \lvert \mathbf{1}_A(\mathbf{y}) - \mathbf{1}_{C_{p_\mathbf{x}}}(\mathbf{y}) \rvert \, \mathrm{d}\omega(\mathbf{y})\\
&= \mathrm{I}_1^{\scriptscriptstyle N} +  \mathrm{I}_2^{\scriptscriptstyle N}.
\end{align*}
The spatial localization bound of the directional wavelets stated in \eqref{wavelets bound 2}, or in \eqref{d_delta} if $K=1$, yields
\begin{align*}
\mathrm{I}_2^{\scriptscriptstyle N} & \leq \int_{N^{-\varepsilon}}^\pi \frac{c_q\,N^2 \sin \theta }{(1+N \theta)^{q+1-K}} \,  \mathrm{d}\theta \leq \int_{N^{-\varepsilon}}^{\pi/2} \frac{c_q\,N^2 \, \theta }{(1+N \theta)^{q+1-K}} \,  \mathrm{d}\theta.
\end{align*}
Thus, using integration by parts, we obtain
\begin{equation*}
\mathrm{I}_2^{\scriptscriptstyle N} \leq c_q \, N^{(1-\varepsilon)(1+K-q)} \leq c_q\, N^{2-4\varepsilon}.
\end{equation*}
Here, the second inequality follows from the assumption that $q\geq  K-3 +2/(1-\varepsilon)$. Additionally, if $ N^{-\varepsilon} < d_\delta/2$ the integral $\mathrm{I}_1^{\scriptscriptstyle N}$ can be estimated using \hyperref[lemma1]{Lemma~\ref*{lemma1}}. We obtain
\begin{equation*}
\mathrm{I}_1^{\scriptscriptstyle N} \leq c \left( \sup_{t \in [a, b]} \| \mathbf{v}'''(t) \|_2 +1 \right)\| \Psi_{\scriptscriptstyle K}^{\scriptscriptstyle N} \|_\infty  N^{-4\varepsilon}.
\end{equation*}
It is not hard to verify that $\| \Psi_{\scriptscriptstyle K}^{\scriptscriptstyle N} \|_\infty = \mathcal{O}(N^2)$ as $N \rightarrow \infty$. Indeed, since $\lvert Y_n^k(\mathbf{x}) \rvert \leq \sqrt{\frac{2n+1}{4\pi}}$, which is a direct consequence of the famous addition theorem for spherical harmonics, the latter follows easily from the definition \eqref{directional wavelets}. Consequently, \eqref{eq15} holds. 
\end{proof}

\begin{proof}[Proof of \autoref{thm1}]
Let $(\mathbf{x}, \mathbf{r}) \in UT\mathbb{S}^2$ with $ \mathbf{x} \in  \mathcal{N}(\mathbf{v}([a, b]_\delta)) $ and let $C_{p_\mathbf{x}}$ be the corresponding spherical cap from \hyperref[lemma1]{Lemma~\ref*{lemma1}}. By \hyperref[lemma4]{Lemma~\ref*{lemma4}}, setting $\varepsilon = 3/4$, we have
\begin{align*}
 &\lvert \langle \mathbf{1}_A, \mathcal{D}(\mathbf{x}, \mathbf{r}) \Psi_{\scriptscriptstyle K}^{\scriptscriptstyle N} \rangle -\langle \mathbf{1}_{C_{p_\mathbf{x}}}, \mathcal{D}(\mathbf{x}, \mathbf{r}) \Psi_{\scriptscriptstyle K}^{\scriptscriptstyle N} \rangle \rvert \nonumber\\
 & \qquad \qquad \qquad \qquad \leq c(\kappa, \phi^\ast, K) \left( \sup_{t \in [a, b]} \| \mathbf{v}'''(t)\|_2 +1\right) N^{-1},
\end{align*}
provided that $N^{-3/4}<d_\delta/2$.
The lower bound now follows from the lower bound for spherical caps in \autoref{thm3}.

Now, let $\varepsilon=1-2(4+\ell)^{-1}$. To prove the upper bound, we first consider the case $d( \mathbf{\mathbf{x}},  \partial A)\geq N^{-\varepsilon}$. If $\mathbf{x} \notin A$, we can, just like in the proof of \hyperref[lemma4]{Lemma~\ref*{lemma4}}, use the localization property of $\Psi_{\scriptscriptstyle K}^{\scriptscriptstyle N} $ to derive
\begin{equation}\label{eq16}
\lvert \langle \mathbf{1}_A, \mathcal{D}(\mathbf{x}, \mathbf{r}) \Psi_{ \scriptscriptstyle K}^{\scriptscriptstyle N} \rangle \rvert \leq  c_q \, N^{(1-\varepsilon)(1+K-q)} \leq \frac{c_{\ell}}{(1+N\pi)^\ell} \leq \frac{c_\ell}{(1+N d( \partial A, \mathbf{x} ))^\ell}.
\end{equation}
If, on the other hand, $\mathbf{x} \in A$, the fact that $\Psi_{\scriptscriptstyle K}^{\scriptscriptstyle N}$ has a zero mean implies
\begin{equation*}
\langle \mathbf{1}_A, \mathcal{D}(\mathbf{x}, \mathbf{r}) \Psi_{\scriptscriptstyle K}^{\scriptscriptstyle N} \rangle = -  \langle \mathbf{1}_{\mathbb{S}^2\setminus A}, \mathcal{D}(\mathbf{x}, \mathbf{r}) \Psi_{\scriptscriptstyle K}^{\scriptscriptstyle N} \rangle
\end{equation*}
and thus \eqref{eq16} still holds. 

Finally, let $d( \mathbf{x}, \partial A) < N^{-\varepsilon}$. One easily verifies that $q\geq  K-3 +2/(1-\varepsilon) = K+1+\ell$. Thus, \hyperref[lemma4]{Lemma~\ref*{lemma4}} yields
\begin{align*}
\lvert \langle \mathbf{1}_A, \mathcal{D}(\mathbf{x}, \mathbf{r}) \Psi_{\scriptscriptstyle K}^{\scriptscriptstyle N} \rangle \rvert \leq   \lvert \langle & \mathbf{1}_{C_{p_\mathbf{x}}}, \mathcal{D}(\mathbf{x}, \mathbf{r}) \Psi_{\scriptscriptstyle K}^{\scriptscriptstyle N} \rangle \rvert  +  c \left( \sup_{t \in [a, b]} \| \mathbf{v}'''(t) \|_2 +1 \right)  N^{2-4\varepsilon}.
\end{align*}
Now, the upper bound in \autoref{thm3} states that
\begin{equation*}
\lvert \langle \mathbf{1}_{C_{p_\mathbf{x}}}, \mathcal{D}(\mathbf{x}, \mathbf{r}) \Psi_{\scriptscriptstyle K}^{\scriptscriptstyle N} \rangle \rvert \leq \frac{c_\ell}{(1+N d(\mathbf{x}, \partial A))^\ell}.
\end{equation*}
Furthermore,
\begin{align*}
N^{2-4\varepsilon} \leq c(1+N^{1-\varepsilon})^{-\ell} \leq c(1+N d(\mathbf{x}, \partial A))^{-\ell},
\end{align*}
where we have used that $\varepsilon = 1-2(4+\ell)^{-1}$. This completes the proof.
\end{proof}

\begin{proof}[Proof of \autoref{thm2}]
For $\varepsilon = 3/4$, \hyperref[lemma4]{Lemma~\ref*{lemma4}} yields
\begin{equation*}
 \langle \mathbf{1}_{A}, \mathcal{D}(\mathbf{x}, \mathbf{r}) \Psi_{\scriptscriptstyle K}^{\scriptscriptstyle N} \rangle  = \langle \mathbf{1}_{C_{p_\mathbf{x}}}, \mathcal{D}(\mathbf{x}, \mathbf{r}) \Psi_{\scriptscriptstyle K}^{\scriptscriptstyle N} \rangle + R_{ \scriptscriptstyle K}^{\scriptscriptstyle N}(\mathbf{x}, \mathbf{r}),
\end{equation*}
where
\begin{equation*}
\lvert R_{ \scriptscriptstyle K}^{\scriptscriptstyle N}(\mathbf{x}, \mathbf{r}) \rvert\leq c(\kappa, \phi^\ast, K)\left(\sup_{t \in [a, b]} \| \mathbf{v}'''(t)\|_2+1\right)N^{-1},
\end{equation*}
provided that $N^{-3/4}<d_\delta/2$. Here, $C_{p_\mathbf{x}}$ is the spherical cap from \hyperref[lemma1]{Lemma~\ref*{lemma1}} with opening angle $\phi = \arcsin(\|\mathbf{v}''(p_\mathbf{x})\|_2^{-1})$ and $d(\mathbf{x}, \partial C_{p_\mathbf{x}}) = d(\mathbf{x}, \partial A)$. Hence, using \eqref{wavelets asympt4}, we obtain the desired formula.
\end{proof}

We conclude this section with the following remark.
\begin{remark}
The estimates in \autoref{thm1} and \autoref{thm2} hold, with minor modifications, also for $C^2$-curves. This can be shown by using exclusively great circles in the approximation of the boundary $\partial A$. In particular, this approach does not take into account the curvature, which leads to a less optimal upper bound of the form
\begin{equation*}
c\left( \sup_{t \in [a, b]} \| \mathbf{v}''(t) \|_2 + 1 \right)  r^{3}
\end{equation*}
in \eqref{eq9}. The verification of this variation of \hyperref[lemma1]{Lemma~\ref*{lemma1}} is analogous to the original proof and, in fact, it is easier. Using this bound, one obtains essentially the same estimates as stated in \autoref{thm1}, but now valid for all sets $A$ with a $C^2$-boundary. Of course, one has to replace  $\mathbf{v}'''$ in \eqref{wavelets asympt3} by $\mathbf{v}''$. However, in the resulting leading term of the asymptotic formula in \autoref{thm2} one has to replace $\| \mathbf{v}''(p_\mathbf{x})\|_2$ by $1$, as unit speed parameterizations of great circles always have constant curvature equal to $1$, which yields a less precise description of the wavelet coefficients. This is reflected in the remainder only being $\mathcal{O}(N^{-\varepsilon})$, $\varepsilon<1$, provided that $q\geq K+1+3\varepsilon(1-\varepsilon)^{-1}$.
\end{remark}

\section{Numerical example}\label{sec5}
In the previous section, we proved that jump discontinuities along smooth curves can be identified, in terms of position and orientation, by the asymptotic behavior of the corresponding frame coefficients in an arbitrary small neighborhood. By Parseval's theorem, for a given signal $f$ it holds that
\begin{equation}\label{eq21}
\langle f,  \mathcal{D}(\mathbf{x}, \mathbf{r}) \Psi_{ \scriptscriptstyle K}^{\scriptscriptstyle N} \rangle = \sum_{n=0}^{2N}\sum_{k=-n}^n \langle f, Y_n^k\rangle \, \overline{\langle \mathcal{D}(\mathbf{x}, \mathbf{r}) \Psi_{ \scriptscriptstyle K}^{\scriptscriptstyle N}, Y_n^k\rangle},
\end{equation}
i.e., the wavelet coefficients are linear combinations of the Fourier coefficients of $f$. Therefore, as a consequence of our main results, equation \eqref{eq21} provides a method to extract precise information about local phenomena using only the global quantities $\langle f, Y_n^k \rangle$. Moreover, \eqref{eq21} can be utilized to compute the wavelet coefficients in practice, where the Fourier coefficients of the signal under consideration are often available.
\begin{figure}
\centering
\includegraphics[ width=0.45\textwidth]{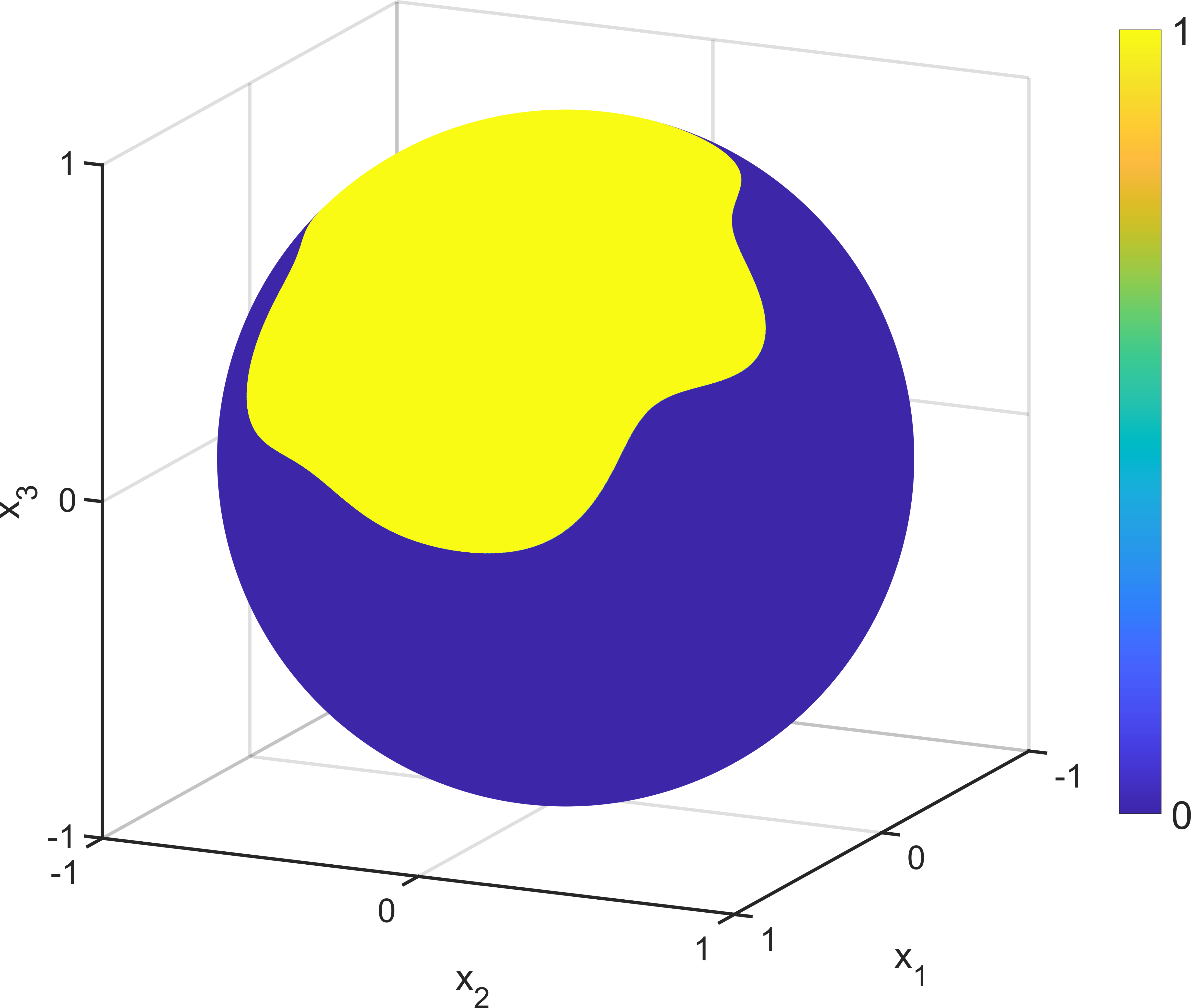}
\caption{Test signal $f = \mathbf{1}_A$}\label{Fig5}
\end{figure}

In this final section, we want to present some numerical experiments and compare the results to our theoretical findings. Let us consider the test signal $f = \mathbf{1}_A$, where $A\subset \mathbb{S}^2$ is the set visualized in \autoref{Fig5}. More precisely, $f(\mathbf{x}) = f_0([\mathbf{R}_{\mathbf{e}_2}(\pi/5)]^{-1}\mathbf{x})$ with
\begin{equation*}
f_0(\theta, \varphi) = \begin{cases}
1, \quad &\theta < g(\varphi),\\
0,  &\theta \geq g(\varphi)
\end{cases}
\end{equation*}
and
\begin{equation*}
g(\varphi) = \frac{7}{500}(20 \pi + 5\cos 2 \varphi + 5 \sin 5 \varphi - 2 \sin 7 \varphi).
\end{equation*}
Additionally we choose $\kappa$ in \eqref{directional wavelets} to be the $C^\infty$ function constructed in \citep{bib9} with $\text{supp}(\kappa) = [1/2, 2]$. As discussed in \citep{bib25}, it holds that
\begin{equation*}
\langle f, \mathcal{D}(\alpha, \beta, \gamma)\Psi_{\scriptscriptstyle K}^{\scriptscriptstyle N} \rangle  = \sum_{n=0}^{2N} \sum_{k=-n}^n \sum_{k'=-n}^n \langle f, Y_n^k \rangle \, \overline{\langle\Psi_{\scriptscriptstyle K}^{\scriptscriptstyle N} , Y_n^{k'} \rangle} \, \overline{D_{k, k'}^n(\alpha, \beta, \gamma)},
\end{equation*}
in terms of the Wigner $D$-functions $D_{k, k'}^n$ and with respect to the Euler angle parameterization of the rotation group. The above sum constitutes the Fourier synthesis of a band-limited function on the $SO(3)$ and thus can be evaluated at arbitrary points $(\alpha_j, \beta_j, \gamma_j)$, $j=1, ..., M$, using fast algorithms. However, like in many practical situations, the exact Fourier coefficients of $f$ are unknown and have to be approximated using only finitely many samples. In our setting, this method inevitably leads to a certain inaccuracy since the signal is not compactly supported in frequency space. 

\begin{figure}
\centering
\includegraphics[ width=0.95\textwidth]{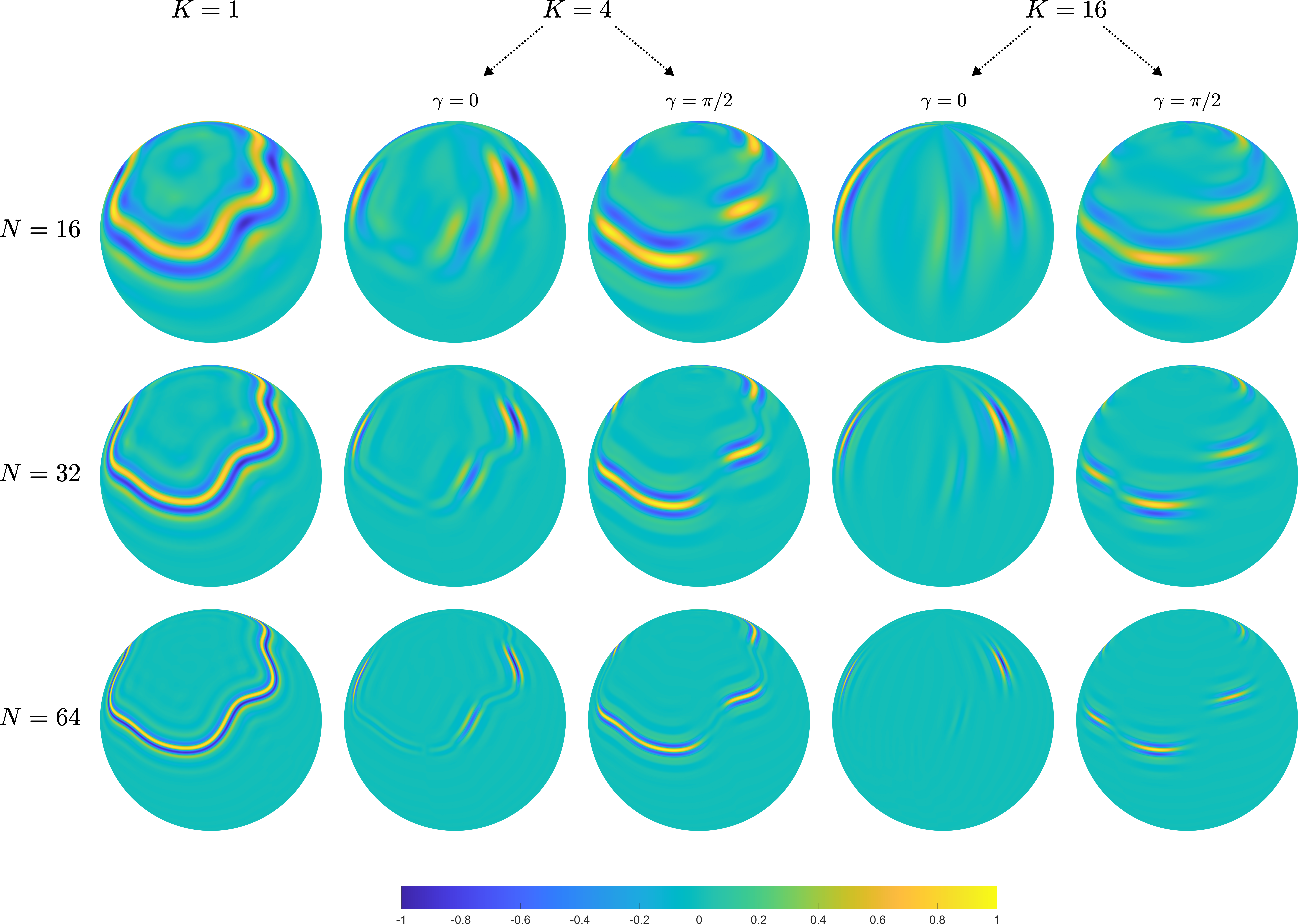}
\caption{Approximate wavelet coefficients $(\beta, \alpha) \mapsto \hat{W}_f^{\scriptscriptstyle K, N}(\alpha, \beta, \gamma)$ for different values $N, K$ and $\gamma$}\label{Fig6}
\end{figure}
\begin{figure}
\center
\includegraphics[width=1\textwidth]{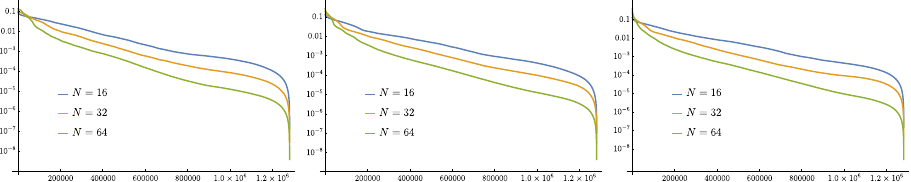}
\caption{Approximate coefficients $\lvert \hat{W}_f^{\scriptscriptstyle K, N}(\alpha_j, \beta_j, \pi/2)\rvert$ plotted in descending order for $K=1$ (left), $K=4$ (middle) and $K=16$ (right)  }\label{Fig7}
\end{figure}

For our experiment, we approximated the integrals $\langle f, Y_n^k\rangle$ by values $\hat{f}_{n, k}$ resulting from a Gauß-Legendre quadrature rule which is exact for spherical polynomials up to degree $2048$. Subsequently, we computed the approximate wavelet coefficients
\begin{equation*}
\hat{W}_f^{\scriptscriptstyle K, N}(\alpha_j, \beta_\ell, \gamma) = \sum_{n=0}^{2N} \sum_{k=-n}^n \sum_{k'=-n}^n \hat{f}_{n, k} \, \overline{\langle\Psi_{\scriptscriptstyle K}^{\scriptscriptstyle N} , Y_n^{k'} \rangle} \, \overline{D_{k, k'}^n(\alpha_j, \beta_\ell, \gamma)},
\end{equation*}
with
\begin{equation*}
\alpha_j = \frac{j \pi}{M}, \; j=1, ..., 2M,\quad \beta_\ell = \frac{\ell \pi}{M}, \; \ell = 1, ..., M,  \quad \gamma \in \{0, \pi/2 \}, \quad M=800,
\end{equation*}
for different values of $K$ and $N$. All computations were done in \texttt{Matlab} using the \texttt{NFFT 3} library \citep{bib26}. In \autoref{Fig6}, the data points $\hat{W}_f^{\scriptscriptstyle K, N}(\alpha_j, \beta_\ell, \gamma)$, $\gamma \in \{0, \pi/2 \}$, are visualized as functions on $\mathbb{S}^2$. Here, for the sake of clarity, all images have been re-scaled to take values between $-1$ and $1$. The magnitudes of the frame coefficients are given in \autoref{Fig7}, where the absolute values $\lvert\hat{W}_f^{\scriptscriptstyle K, N}(\alpha_j, \beta_\ell, \gamma)\rvert$ are plotted in descending order. It is immediately noticeable that the coefficients $\hat{W}_f^{\scriptscriptstyle K, N}$ become more and more concentrated along the boundary $\partial A$ as $N$ increases. Moreover, the peaks remain stable in their magnitude and shape and, essentially, appear to just get dilated for larger values of $N$. The directional sensitivity increases with $K$ and the area of the boundary getting detected changes with different orientations of the wavelets. Also, the parameter $N$ does not seem to affect the directional sensitivity. Hence, our observations in this numerical example are in accordance with the characteristic properties of the exact wavelet coefficients displayed in \autoref{thm1} and \autoref{thm2}.

\bibliographystyle{elsarticle-harv}
\bibliography{mybib}

\end{document}